\documentclass[reqno]{amsart}
\usepackage[hmarginratio=1:1, vmarginratio=1:1]{geometry}
\usepackage{pdfsync}
\usepackage[dvipsnames]{xcolor}
\usepackage{amsmath}
\usepackage{amsfonts}
\usepackage{amssymb}

\usepackage{amsthm}
\usepackage{mathrsfs}
\usepackage{hyperref}
\hypersetup{
    colorlinks=true,
    linkcolor=blue,
    citecolor=red,
}
\usepackage{theoremref}

\usepackage{tikz}
\usetikzlibrary{shapes,arrows,cd}
\usetikzlibrary{calc,arrows.meta,positioning}
\usetikzlibrary{decorations.pathreplacing}
\usetikzlibrary{decorations.markings}
\usepackage{float}

\usepackage[utf8]{inputenc}
\usepackage{array}
\usepackage[inline]{enumitem}
\usepackage{verbatim}
\usepackage{comment}

\theoremstyle{plain}
\newtheorem{lem}{Lemma}[section]
\newtheorem{thm}[lem]{Theorem}
\newtheorem{cor}[lem]{Corollary}
\newtheorem{claim}{Claim}[lem]

\newtheorem{prop}[lem]{Proposition}
\newtheorem*{thm*}{Theorem}
\newtheorem{thmAB}{Theorem}

\theoremstyle{definition}
\newtheorem{defn}[lem]{Definition}
\newtheorem{rmk}[lem]{Remark}
\newtheorem{ex}[lem]{Example}
\newtheorem{question}[lem]{Question}

\makeatletter
 \newcommand{\myhyper}[1]{\Hy@raisedlink{\hypertarget{#1}{}}}
\makeatother

\DeclareMathOperator{\inv}{inv}
\DeclareMathOperator{\PIG}{PIG}
\DeclareMathOperator{\FI}{FI}
\DeclareMathOperator{\FIM}{FIM}
\DeclareMathOperator{\id}{id}

\let\P\relax
\DeclareMathOperator{\P}{\textup{\textsf{P}}} 
\DeclareMathOperator{\N}{\textup{\textsf{N}}} 
\DeclareMathOperator{\E}{\textup{\textsf{E}}}
\DeclareMathOperator{\PE}{\textup{\textsf{PE}}}
\DeclareMathOperator{\NE}{\textup{\textsf{NE}}}

\DeclareMathOperator{\B}{\textup{\textbf{B}}}

\DeclareMathOperator{\W}{\textup{\textsf{W}}}

\DeclareMathOperator{\V}{\textup{\textsf{V}}}

\newenvironment{thmenumerate}{\begin{enumerate}[topsep=0mm, leftmargin=14mm,itemsep=1mm,label=\textup{(\roman*)}]
}{\end{enumerate}}

\newenvironment{bulletmidsentence}{\begin{itemize}[topsep=0mm, leftmargin=8mm, itemsep=1mm]
}{\end{itemize}}

\newenvironment{rmkenumerate}{\begin{enumerate}[align=left, leftmargin = 0mm, listparindent=0mm, labelwidth=2mm, itemindent=!, itemsep=1.5mm, labelsep=3mm, label=\textup{(\arabic*)}]
}{\end{enumerate}}

\numberwithin{equation}{section}

\begin{document}

\title[Howson property and finitely generated intersection problem]{Howson property and finitely generated intersection problem for monogenic inverse semigroups}
	
\author{Jung Won Cho}
	
\address{School of Mathematics and Statistics, University of St Andrews, St Andrews, Fife KY16 9SS, UK.}
\email{jwc21@st-andrews.ac.uk}
	
\author{Craig Miller}
\address{Department of Computer Science, Durham University, Stockton Road, Durham, DH1 3LE, UK}
\email{craig.a.miller@durham.ac.uk}

\author{Nik Ru\v{s}kuc}
\address{School of Mathematics and Statistics, University of St Andrews, St Andrews, Fife KY16 9SS, UK.}
\email{nik.ruskuc@st-andrews.ac.uk}
	


\begin{abstract}
An algebraic structure is said to have the Howson property if the intersection of any two finitely generated subalgebras is finitely generated. We explore the Howson property in the context of monogenic inverse semigroups. It is known, due to work of Jones and Trotter (1989) and Jones (2016), that every monogenic inverse semigroup has the Howson property considered as an inverse semigroup, i.e.\ with respect to its inverse subsemigroups. In this paper, we consider monogenic inverse semigroups qua semigroups, i.e.\ we consider \emph{all} their subsemigroups. We prove that every monogenic inverse semigroup possesses the Howson property in this broader sense, with the sole exception of the monogenic free inverse semigroup. For this exceptional case, we show that the problem of determining whether the intersection of two finitely generated subsemigroups is finitely generated is algorithmically decidable.
\end{abstract}

\maketitle

\section{Introduction}
\label{sec:Intro}

The starting point for this paper is the following question regarding a given monogenic inverse semigroup $S$: \emph{Is the intersection of two finitely generated subsemigroups of $S$ always finitely generated?} For those $S$ where the answer is negative, we proceed to ask: \emph{Is it algorithmically decidable whether such an intersection is finitely generated?}

An algebraic structure is said to have the \emph{Howson property} if the intersection of any two finitely generated subalgebras is again finitely generated. This property is named after A. G. Howson, who proved that the intersection of any two finitely generated subgroups of a free group is finitely generated \cite{Howson}. This theorem motivated further investigation into the rank of such intersections, leading to the \emph{Hanna Neumann conjecture}, posed in \cite{Neumann} and later proved independently by J.\ Friedman \cite{Friedman} and I. Mineyev \cite{Mineyev}. Further examples of groups with the Howson property include: 
groups whose finitely generated subgroups are Noetherian, for example finite, abelian, nilpotent and virtually polycyclic groups;
limit groups \cite{Dahmani03}; fundamental groups of compact surfaces \cite{Greenberg60};
free products of groups with the Howson property \cite{Baumslag}.

Slightly surprisingly, the Howson property for semigroups has received little attention. However, it is known that the monogenic free semigroup $\mathbb{N}$ has the property (for trivial reasons), but that free semigroups of higher rank do not, see for example \cite{BlattnerHead}.

On the other hand, the Howson property has received considerable attention in the context of inverse semigroups.
An \emph{inverse semigroup} $S$ is a semigroup in which for every $s\in S$ there is a unique $s^{-1}\in S$ such that $ss^{-1}s = s$ and $s^{-1}ss^{-1}=s^{-1}$. It can also be viewed as an algebra with an associative binary operation and a unary operation $^{-1}$ satisfying the following laws: for all~$x,y\in S$,
\[
(x^{-1})^{-1} = x, \quad\; (xy)^{-1} = y^{-1}x^{-1},\quad\; xx^{-1}x = x,\quad\; xx^{-1}yy^{-1} = yy^{-1}xx^{-1}.
\]
An inverse semigroup can be regarded as a (plain) semigroup by disregarding the unary operation. The meaning of \emph{subsemigroups} changes between those two viewpoints: an \emph{inverse} subsemigroup of $S$ is closed under multiplication and inverses, whereas a \emph{plain} subsemigroup need only be closed under multiplication. Of course, every inverse subsemigroup is a plain subsemigroup as well, but the converse is not true.

This in turn gives rise to two different Howson properties for an inverse semigroup $S$. We say that $S$ has:
\begin{bulletmidsentence}
\item the \emph{inverse semigroup Howson property} if the intersection of any two finitely generated inverse subsemigroups of $S$ is finitely generated;
\item the \emph{semigroup Howson property} if the intersection of any two finitely generated plain subsemigroups of $S$ is finitely generated.
\end{bulletmidsentence}

We note that the intersection of two subsemigroups of $S$ may be empty. Throughout this paper, we regard the empty set as a finitely generated subsemigroup.

In the context of the inverse semigroup Howson property, the following result of Jones and Trotter \cite{JonesTrotter} is of particular significance for us:

\begin{thm}[\cite{JonesTrotter}]
\label{thm:JonesTrotter}
The monogenic free inverse semigroup has the inverse semigroup Howson property. \hfill\qed
\end{thm}

In contrast with free groups, free inverse semigroups of higher rank do not have the Howson property \cite{JonesTrotter}. Subsequently, Jones \cite{Jones} extended the above result to all monogenic inverse semigroups: 

\begin{thm}[\cite{Jones}]
\label{thm:Jones}
Every monogenic inverse semigroup has the inverse semigroup Howson property. \hfill\qed
\end{thm}

In the same paper, Jones \cite{Jones} also studied this property in the context of inverse semigroups with finitely many idempotents. Moreover, Silva and Soares \cite{SilvaSoares} investigated it for semidirect products of semilattices by groups.

Motivated by Theorem~\ref{thm:Jones}, a natural question arises as to whether monogenic inverse semigroups have the \emph{semigroup} Howson property. Our first main result provides a complete answer to this question.

\begin{thmAB}
\label{thmA:Howson}
A monogenic inverse semigroup has the semigroup Howson property if and only if it is not free.
\end{thmAB}

The proof of this result will be divided into two parts: Theorem~\ref{thm:HonwsonNonFree} for non-free monogenic inverse semigroups, and Example~\ref{ex:HowsonCounter1} showing that the monogenic free inverse semigroup $\FI_1$ is not (semigroup) Howson.

In the case of $\FI_1$, the result above naturally gives rise to an algorithmic decision problem:
is it decidable whether the intersection of two finitely generated subsemigroups of $\FI_1$ is finitely generated? 
We formulate this problem in a general semigroup context as follows.
\medskip

\noindent
\textbf{Finitely Generated Intersection Problem (FGIP) for a semigroup $U$:} Is it algorithmically decidable whether the intersection of two given finitely generated subsemigroups $S$ and $T$ of $U$ is finitely generated?
\medskip

Our second main result is:

\begin{thmAB}
\label{thmB:FGIP}
The finitely generated intersection problem for the monogenic free inverse semigroup as a semigroup is decidable.
\end{thmAB}

The paper is structured as follows. 
Section~\ref{sec:Prelim} consists of preliminary definitions and results, including the classification of monogenic inverse semigroups. 
Section~\ref{sec:NonFree} establishes the converse implication of Theorem~\ref{thmA:Howson}, namely that every non-free monogenic inverse semigroup has the semigroup Howson property.
From Section~\ref{sec:FI1Prelim}, our focus is on the monogenic free inverse semigroup $\FI_1$. We begin with notation and basic results regarding $\FI_1$ in Section~\ref{sec:FI1Prelim}. 
The forward implication of Theorem~\ref{thmA:Howson} is then proved in Section~\ref{sec:FI1NonHowson}.
Towards proving Theorem~\ref{thmB:FGIP}, 
we first obtain characterisations of finitely generated subsemigroups of $\FI_1$ in
Section~\ref{sec:Character}. Section~\ref{sec:FGIP} is devoted to the proof of Theorem~\ref{thmB:FGIP}. We finish the paper with remarks and open questions in Section~\ref{sec:RemarkOpen}.

\section{Preliminaries}
\label{sec:Prelim}

In this section, we cover the general notation and concepts that will be used throughout the paper. 
Further preliminaries will be presented in Section~\ref{sec:FI1Prelim}, where the monogenic free inverse semigroup and its key properties will be introduced.

\subsection{General notation} 
We use the standard notation  $\mathbb{Z}$, $\mathbb{N}$ and $\mathbb{N}_0$ to denote the sets of integers, positive integers and non-negative integers, respectively. 
We will also use the same notation for the corresponding additive semigroups, i.e. $\mathbb{Z}$ for the free cyclic group, $\mathbb{N}$ for the free semigroup, and $\mathbb{N}_0$ for the free monoid.
For $m,n\in\mathbb{Z}$ with $m\leq n$, we use the interval notation $[m,n]$ to denote the set $\{m, m+1, \dots, n\}$, and abbreviate this to just $[n]$ when $m=1$. 

\subsection{Semigroups and inverse semigroups}
For basic semigroup theory, we refer the reader to the standard texts \cite{CliffordPreston1, Howie}. For inverse semigroups, the reader may consult \cite{Lawson, Petrich}.
We will make occasional use of Green's $\mathscr{R}$, $\mathscr{L}$, $\mathscr{H}$, $\mathscr{D}$ equivalences on a semigroup $S$, which are defined by:
\[
x\mathscr{R}y \Leftrightarrow xS^1=yS^1,\;\;\; x\mathscr{L}y \Leftrightarrow S^1x=S^1y,\;\;\; 
\mathscr{H}=\mathscr{R}\cap\mathscr{L},\;\;\; 
\mathscr{D}=\mathscr{R}\circ\mathscr{L}(=\mathscr{R}\circ\mathscr{L}).
\]
Here $S^1$ stands for $S$ with an identity adjoined to it if $S$ does not already have one. For an inverse semigroup $S$, we have a simpler description of Green's $\mathscr{R}$- and $\mathscr{L}$-relations: 
\[
x\mathscr{R}y \Leftrightarrow xx^{-1} =  yy^{-1}, \;\;\; x\mathscr{L}y \Leftrightarrow x^{-1}x = y^{-1}y. 
\]

For a semigroup $S$ we denote by $E(S)$ the set of its idempotents. When $S$ is inverse, $E(S)$ forms a semilattice, i.e. a commutative semigroup of idempotents.

Throughout the paper, the term \emph{subsemigroup} will refer to a subset of a semigroup $S$ that is closed under products. When we require closure under inverses as well (in an inverse semigroup) we will use the term \emph{inverse subsemigroup}.

\subsection{Monogenic inverse semigroups} 
The complete classification of monogenic inverse semigroups seems to have been (re)discovered several times, in slightly different forms, 
first by authors in the former Soviet Union and later elsewhere.
The earliest such classification seems to be in papers by Gluskin \cite{Gluskin57,Gluskin61,Gluskin63}, and subsequent papers include \cite{Ershova71,Djadcenko73,Djadcenko74,Eberhart72,Conway84,Preston86}. 
A detailed account of the classification can be found in
\cite[Chapter IX]{Petrich}, 
which we follow here.

\begin{thm}[{\cite[Chapter IX]{Petrich}}]
\label{thm:MonoClass1}
A monogenic inverse semigroup is isomorphic to one of the following:
\begin{thmenumerate}
    \item \label{it:corclass1} the monogenic free inverse semigroup $\FI_1$;
    \item \label{it:corclass2} a finite inverse semigroup;
    \item \label{it:corclass3} an inverse semigroup $S$ with an ideal $I$ isomorphic to the bicyclic semigroup $\B$ such that $S\setminus I$ is finite;
    \item \label{it:corclass4} an inverse semigroup $S$ with an ideal $I$ isomorphic to the integers $\mathbb{Z}$ such that $S\setminus I$ is finite. \hfill\qed
\end{thmenumerate}
\end{thm}

In \cite[Theorem IX.3.11]{Petrich} each of the types appearing in the above classification
is explicitly given by means of an inverse semigroup presentation. Specifically, $\langle a\mid \ \rangle$ for \ref{it:corclass1},
$\langle a\mid a^n = a^{n+m} \rangle$ ($m,n\in\mathbb{N}$) for \ref{it:corclass2},
$\langle a\mid a^n = a^{-1}a^{n+1} \rangle$ ($n\in\mathbb{N}$) for \ref{it:corclass3}, and
$\langle a\mid a^na^{-1} = a^{-1}a^{n} \rangle$ ($n\in\mathbb{N}$) for \ref{it:corclass4}.
However, we will not require this level of detail, and will only need workable definitions for $\FI_1$ and $\B$, which will be given in Sections \ref{sec:FI1NonHowson} and \ref{sec:NonFree}, respectively.

\subsection{Finite generation auxiliary results} The following elementary results concerning finite generation and the (inverse) semigroup Howson property will be used throughout the paper. 

\begin{prop}
\label{prop:invplaingen}
An inverse semigroup is finitely generated as an inverse semigroup if and only if it is finitely generated as a (plain) semigroup. \hfill\qed
\end{prop}

\begin{prop}
\label{prop:unionfg}
A semigroup that is a finite union of finitely generated subsemigroups is finitely generated. \hfill\qed
\end{prop}

\begin{prop}
\label{prop:subHowson}
If a semigroup (resp.\ inverse semigroup) has the semigroup Howson property (resp.\ inverse semigroup Howson property), then so do its subsemigroups (resp.\ inverse subsemigroups). \hfill\qed
\end{prop}

\begin{prop}[\cite{SitSiu}]
\label{prop:SitSiu}
Any subsemigroup of $\mathbb{N}_0$ is finitely generated. \hfill\qed
\end{prop}

\begin{prop}[{\cite{Jura}, \cite[Theorem 1.1]{RuskucLarge}}]
\label{prop:FiniteGenerationLarge}
Let $S$ be a semigroup and let $T$ be a subsemigroup of $S$ such that $S\setminus T$ is finite. Then, $S$ is finitely generated if and only if $T$ is finitely generated. \hfill\qed
\end{prop}

\section{Non-free monogenic inverse semigroups}
\label{sec:NonFree}

In this section, we show that every non-free monogenic inverse semigroup has the semigroup Howson property; these are semigroups listed in  \ref{it:corclass2}, \ref{it:corclass3} and \ref{it:corclass4} of Theorem~\ref{thm:MonoClass1}. 
We begin by considering the bicyclic semigroup $\B$. For an introduction to this well-known semigroup, we refer the reader to \cite[Chapter 1]{Howie} or \cite[Chapter 3]{Lawson}.

We will be working with the following concrete representation of $\B$:
\[
\B=\mathbb{N}_0\times\mathbb{N}_0,\quad 
(i,j)(k,l)=
\begin{cases}
(i-j+k,l) & \text{ if } j \leq k \\
(i,j-k+l) & \text{ if } j > k.
\end{cases}
\]

Descal\c{c}o and Ru\v{s}kuc \cite{DescalcoRuskuc2005, DescalcoRuskuc2008} described all (finitely generated) subsemigroups of $\B$, and we make use of some of their results.

The bicyclic semigroup $\B$ has a single $\mathscr{D}$-class and hence admits the following egg-box representation, in which rows and columns correspond to $\mathscr{R}$- and $\mathscr{L}$-classes, respectively: 
\vspace{1mm}
{\begin{center}
\begin{tikzpicture}[scale=0.55]
\foreach \i in {0, 1, 2, 3, 4} {
\draw[line width=1pt] (0, 1.5*\i) -- (6, 1.5*\i);
\draw[line width=1pt] (1.5*\i, 0) -- (1.5*\i, 6);}
\foreach \i in {0, 1, 2, 3, 4} {
\draw[line width=1pt] (6, 1.5*\i) -- (6.25, 1.5*\i);
\draw[line width=1pt] (1.5*\i, 0) -- (1.5*\i, -0.25);}

\node (id) at (0.75, 5.18) {\tiny$(0,0)$};
\node (x) at (2.25, 5.18) {\tiny$(0,1)$};
\node (x2) at (3.75, 5.18) {\tiny$(0,2)$};
\node (x3) at (5.25, 5.18) {\tiny$(0,3)$};

\foreach \i in {1, 2, 3}{
\node (x0x\i) at ($(0.75, 4.5)- (0, 1.5*\i) + (0, 0.75)$) {\tiny$(\i,0)$};
}

\foreach \i in {1, 2, 3}{
\node (x1x\i) at ($(2.25, 4.5)- (0, 1.5*\i) + (0, 0.75)$) {\tiny$(\i,1)$};
}

\foreach \i in {1, 2, 3}{
\node (x2x\i) at ($(3.75, 4.5) - (0, 1.5*\i) + (0, 0.75)$) {\tiny$(\i,2)$};
}

\foreach \i in {1, 2, 3}{
\node (x3x\i) at ($(5.25, 4.5) - (0, 1.5*\i) + (0, 0.75)$) {\tiny$(\i,3)$};

\node (dotsh) at (7,3) {\ldots};
\node (dotsv) at (3, -0.75) {\vdots};
}

\end{tikzpicture}
\end{center}
}

In this diagram, the elements on the diagonal are precisely the idempotents of $\B$. 
For every $i\in \mathbb{N}_0$, we write $R_i$ (resp.\ $L_i$) for the $\mathscr{R}$- (resp.\ $\mathscr{L}$-) class consisting of elements in the $(i+1)$th row (resp.\ column).

We call a subsemigroup $S$ of $\B$: 
\begin{bulletmidsentence}
    \item \emph{diagonal} if it solely consists of elements lying on the diagonal, i.e.\ if it is a subsemilattice of the semilattice $E(\B)$ of idempotents of $\B$;
    \item \emph{upper (resp.\ lower)} if it contains no elements below (resp.\ above) the diagonal;
    \item \emph{square} if it has elements both above and below the diagonal.
\end{bulletmidsentence}
We remark that a square subsemigroup was termed two-sided in \cite{DescalcoRuskuc2005, DescalcoRuskuc2008}, but we reserve the term for a different future use in this paper.

We will make use of the following two results:

\begin{prop}[{\cite[Theorem 7.6]{DescalcoRuskuc2005}}]
\label{prop:bisquare}
A square subsemigroup of the bicyclic semigroup $\B$ is a finite union of copies of $\B$ and subsemigroups of $\mathbb{N}_0$. \hfill\qed
\end{prop}

\begin{prop}[{\cite[Theorem 2.1]{DescalcoRuskuc2008}}]
\label{prop:bifinite}
A subsemigroup $S$ of the bicyclic semigroup $\B$ is finitely generated if and only if one of the following holds: 
\begin{thmenumerate}
    \item \label{it:bi1} $S$ is a finite diagonal subsemigroup;
    \item \label{it:bi2} $S$ is a square subsemigroup;
    \item \label{it:bi3} $S$ is an upper subsemigroup and the set $\{i\in\mathbb{N}_0 : R_i\cap S\neq\emptyset
    \}$ is finite;
    \item \label{it:bi4} $S$ is a lower subsemigroup and the set $\{i\in\mathbb{N}_0 : L_i\cap S\neq\emptyset\}$ is finite. \hfill\qed
\end{thmenumerate}
\end{prop}

\begin{rmk}
Parts \ref{it:bi3} and \ref{it:bi4} above correct a couple of obvious errors in the original  statement.
\end{rmk}

\begin{cor}
\label{cor:biHowson}
The bicyclic semigroup $\B$ has the semigroup Howson property.
\end{cor}
\begin{proof}
Let $S$ and $T$ be finitely generated subsemigroups of $\B$. If either $S$ or $T$ satisfies \ref{it:bi1}, \ref{it:bi3} or \ref{it:bi4} of Proposition~\ref{prop:bifinite}, then so does the intersection $S\cap T$ and the result follows. 

Suppose now that both $S$ and $T$ are square subsemigroups. By Proposition~\ref{prop:bisquare}, we may write 
\[
S = \Biggl(\bigcup_{i=1}^n B_i\Biggr) \cup \Biggl(\bigcup_{i=1}^m N_i\Biggr)\quad \text{and}\quad T = \Biggl(\bigcup_{i=1}^k B_i'\Biggr)\cup\Biggl(\bigcup_{i=1}^l N_i'\Biggr),
\]
where all $B_i$, $B_i'$ are copies of $\B$, and all $N_i$, $N_i'$ are subsemigroups of $\mathbb{N}_0$. 

If each intersection $B_i\cap B_j'$ is empty, then $S\cap T$ is a finite union of subsemigroups of $\mathbb{N}_0$ and hence
is finitely generated by Propositions \ref{prop:unionfg} and \ref{prop:SitSiu}. So, suppose that $B_i\cap B_j'\neq\emptyset$ for some $i,j$, and let $(p,q)\in B_i\cap B_j'$. Without loss of generality assume that $p\leq q$. In fact, we claim we can assume without loss of generality that $p<q$. Indeed, if $p=q$,
there exist $r,s>p$ such that $(p,r)\in B_i$ and $(p,s)\in B_j'$, and we can take the element
\[
\bigl(p,p+(r-p)(s-p)\bigr)=(p,r)^{s-p}=(p,s)^{r-p}\in B_i\cap B_j',
\]
instead of $(p,q)$. Now, since $(p,q)\in B_i\cap B_j'$ and both $B_i$ and $B_j'$ are copies of $\B$, it follows that also $(q,p)\in B_i\cap B_j'$. But this in turn means that $S\cap T$ is square, and hence finitely generated 
by Proposition~\ref{prop:bifinite} \ref{it:bi2}. This completes the proof.
\end{proof}

To enable us to deal with finite extensions, we also need the following:

\begin{prop}
\label{prop:HowsonFiniteComplement}
Let $S$ be a semigroup and let $T$ be a subsemigroup of $S$ such that $S\setminus T$ is finite. Then, $S$ has the semigroup Howson property if and only if $T$ does.
\end{prop}
\begin{proof}
The forward direction follows from Proposition \ref{prop:subHowson}. For the converse direction, let $U_1$ and $U_2$ be finitely generated subsemigroups of $S$. For each $i$, since $U_i = (U_i\cap T)\cup (U_i\cap (S\setminus T))$ and $U_i\cap (S\setminus T)$ is finite, $U_i\cap T$ is finitely generated by Proposition~\ref{prop:FiniteGenerationLarge}. Thus, $U_1\cap U_2\cap T$ is finitely generated by assumption. 
Once again, by Proposition~\ref{prop:FiniteGenerationLarge}, 
$U_1\cap U_2$ is finitely generated since $U_1\cap U_2\cap (S\setminus T)$ is finite. This completes the proof.
\end{proof}

\begin{thm}
\label{thm:HonwsonNonFree}
Every non-free monogenic inverse semigroup has the semigroup Howson property.
\end{thm}

\begin{proof}
Let $S$ be a non-free monogenic inverse semigroup. 
By Theorem~\ref{thm:MonoClass1}, $S$ is either finite or a finite extension of $\B$ or $\mathbb{Z}$. The finite case is trivial. Corollary~\ref{cor:biHowson} shows that $\B$ is Howson, while $\mathbb{Z}$ is Howson since its subsemigroups are either isomorphic to $\mathbb{Z}$ or are subsemigroups of $\mathbb{N}_0$. The result now follows from Proposition~\ref{prop:HowsonFiniteComplement}.
\end{proof}

\section{Preliminaries II: the monogenic free inverse semigroup}
\label{sec:FI1Prelim}
Our next task is to show that the monogenic free inverse semigroup $\FI_1$ does not have the semigroup Howson property. To prepare the ground, in this section we introduce preliminary material on $\FI_1$, which is then used to establish the result in Example \ref{ex:HowsonCounter1} in the next section, which in turn completes the proof of Theorem \ref{thmA:Howson}.

Munn~\cite{Munn} showed that every element of a free inverse semigroup can be uniquely associated with an edge-labelled birooted directed tree called a \emph{Munn tree}. 
In the monogenic case, this description reduces to just a non-empty finite directed path 
with two of its vertices designated as the initial and terminal vertices. This can be visualised as follows: 

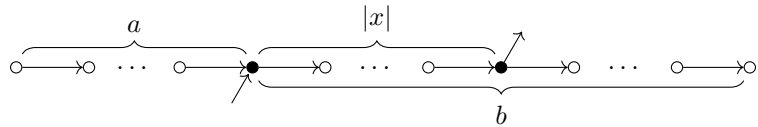
\begin{center}
\begin{tikzpicture}[main/.style = {draw, circle, inner sep = 1.5pt, minimum size=1pt}]
\node[main] (1) {};
\node[main] (2) [right = 0.8cm of 1] {};
\node[main] (3) [right = 2cm of 1] {};
\node[main, fill = black] (4) [right= 0.8cm of 3] {};
\node[main] (5) [right = 0.8cm of 4] {};
\node[main] (6) [right = 1.2cm of 5] {};
\node[main, fill=black] (7) [right = 0.8cm of 6] {};
\node[main] (8) [right = 0.8cm of 7] {};
\node[main] (9) [right = 1.2 of 8] {};
\node[main] (10) [right = 0.8cm of 9] {};	
\node at ($(2)!.5!(3)$) {\ldots};
\node at ($(5)!.5!(6)$) {\ldots};
\node at ($(8)!.5!(9)$) {\ldots};	

\draw[->] (1) -- (2);

\draw[->] (3) -- (4);	

\draw[->] (4) -- (5);		

\draw[->] (6) -- (7);

\draw[->] (7) -- (8);

\draw[->] (9) -- (10);
\draw[->] ($(4)-(60:5.5mm)$) -- ($(4)-(60:1mm)$);
\draw[->] (7) -- ($(7)+(60:5.5mm)$);
\draw [decorate,decoration={brace,amplitude=5pt,raise=5pt}] (1) -- (4) node[midway,above,yshift=10pt]{$a$};
\draw [decorate,decoration={brace,amplitude=5pt,raise=5pt}] (4) -- (7) node[midway,above,yshift=10pt]{$\vert x \vert$};
\draw [decorate,decoration={brace,amplitude=5pt,mirror, raise=5pt}] (4) -- (10) node[midway,above,yshift=-25pt]{$b$};
\end{tikzpicture}
\end{center}
The vertices with in- and out-arrows are the initial and terminal vertices, respectively. It is permitted for the terminal vertex to be to the left of the initial vertex or for the two of them to coincide. The label of an edge in the direction (resp.\ opposite direction) of the arrow is the generator (resp.\ the inverse of the generator) of $\FI_1$. It is customary to omit the labels for the monogenic case.

The product of two Munn trees $M_1$ and $M_2$ of elements of $\FI_1$ is obtained as follows. We overlay $M_1$ and $M_2$ by identifying the terminal vertex of $M_1$ with the initial vertex of $M_2$ and further identify overlapped vertices and edges which arise as a consequence. The initial and terminal vertices of the product are the initial vertex of $M_1$ and the terminal vertex of $M_2$, respectively.

The Munn tree above suggests a parameterisation of elements of $\FI_1$ into triples. Specifically, we view $\FI_1$ as 
\[
\FI_1 := \{(-a, x, b)\in \mathbb{Z}^3 : a,b\geq 0, \,
a+b>0,\, -a\leq x\leq b\}
\]
with multiplication
\[
(-a_1, x_1, b_1)(-a_2, x_2, b_2) = (-\max(a_1, a_2-x_1), \, x_1+x_2,\, \max(b_1, b_2+x_1)).
\]
The inverse of a typical triple $(-a, x, b)$ is $(-(a+x), -x, b-x)$. The triple $(0, 1, 1)$ 
(freely) generates $\FI_1$ as an inverse semigroup.

Adjoining the identity to $\FI_1$ yields the \emph{monogenic free inverse monoid} $\FIM_1$. We may regard the adjoined identity as the triple $(0, 0, 0)$, the Munn tree consisting of a single vertex with no edge.

We review Green's relations for $\FI_1$: 
\begin{align*}
    (-a, x, b) \mathscr{R} (-a', x', b') &\Leftrightarrow a = a' \text{ and } b = b';\\
    (-a, x, b) \mathscr{L} (-a', x', b') &\Leftrightarrow a + x = a' + x' \text{ and } b - x = b' - x';\\
    (-a, x, b) \mathscr{D} (-a', x', b') &\Leftrightarrow a+b = a'+b';\\
    (-a, x, b) \mathscr{H} (-a', x', b') &\Leftrightarrow (-a, x, b) = (-a', x', b').
\end{align*}

The $\mathscr{D}$-classes of $\FI_1$, indexed by natural numbers, are
\[
D_n := \{(-a, x, b)\in\FI_1 : a+b=n\} \quad (n\in\mathbb{N}).
\]
For each $n\in\mathbb{N}$, the union $I_n := \bigcup_{i\geq n}D_i$ forms an ideal of $\FI_1$.

Consider an arbitrary element $(-a, x, b)\in \FI_1$. Observe that this element is an idempotent if and only if $x=0$. We say that this element is \emph{positive} if $x>0$ and \emph{negative} if $x<0$. In terms of Munn trees, idempotents are precisely those whose inital and terminal vertices coincide; positive (resp.\ negative) elements are those whose terminal vertices are on the right (resp.\ left) hand side of the initial vertices. We often write $(-a, x, x+b)$ when we wish to emphasise that an element is positive; similarly, $(-(a+x), -x, b)$ for negative.

We define some subsets of $\FI_1$ that will be used in the remainder of the paper: 
\begin{align*}
    &\E := \{(-a, x, b)\in\FI_1: x=0\};\\
    &\P := \{(-a, x, b)\in\FI_1: x>0\}; &&\N := \{(-a, x, b)\in\FI_1: x<0\};\\
    &\PE := \P\cup \E; &&\NE := \N\cup\E.
\end{align*}
The last four
are further decomposed into subsets 
parametrised by $a,b\in\mathbb{N}_0$ as follows:
\begin{align*}
    &\P_{a,b} := \{(-a, x, x+b) \in \FI_1 : x>0\};&&\N_{a,b} := \{(-(a+x), -x, b) \in\FI_1 : x>0\};\\
    &\PE_{a,b} := \{(-a, x, x+b)\in \FI_1 : x\geq 0\};&
    &\NE_{a,b} := \{(-(a+x), -x, b)\in\FI_1: x\geq 0\}.
\end{align*}

Note that, as $(0, 0, 0)\notin \FI_1$, we have 
\[
\P_{0,0} = \PE_{0,0} = \{(0, x, x): x>0\} \quad\text{and}\quad \N_{0, 0} = \NE_{0,0} = \{(-x, -x, 0): x>0\}.
\]

For any subset $L\subseteq \FI_1$, we define:
\begin{align*}
    \P(L) := L\cap \P; \qquad \N(L) := L\cap \N; \qquad
    \PE(L) := L\cap \PE; \qquad \NE(L) := L\cap \NE.
\end{align*}

The set $\E$ of idempotents of $\FI_1$ is partially ordered by 
\begin{equation}
\label{eq:pord}
(-a, 0, b)\leq (-c, 0, d)\ \Leftrightarrow\ a\geq c\text{ and } b\geq d,
\end{equation}
and as such is isomorphic to $\mathbb{N}_0\times\mathbb{N}_0\setminus\{(0,0)\}$ under the component-wise $\geq$. 
This order coincides with the natural partial order on the idempotents (see \cite[Chapter 1]{Lawson}) but we will not need that fact.

We define three functions: 
\begin{align*}
\lambda \colon \FI_1 &\to \mathbb{N}_0, &\mu \colon \FI_1 &\to \mathbb{Z}, &\rho \colon \FI_1 &\to \mathbb{N}_0,\\
(-a, x, b) &\mapsto a, &(-a, x, b) &\mapsto x, &(-a, x, b) &\mapsto b. 
\end{align*}
Observe that $\mu$ is the natural homomorphism from $\FI_1$ to its maximal group image $\mathbb{Z}$.

Let $S$ be a subsemigroup of $\FI_1$. We say that $S$ is \emph{one-sided} if $S\subseteq \PE$ or $S\subseteq\NE$. We say that $S$ is \emph{two-sided} if both $\P(S)$ and $\N(S)$ are non-empty. Equivalently, $S$ is one-sided if and only if $S\mu$ is a subsemigroup of $\mathbb{N}_0$ or of $-\mathbb{N}_0$; $S$ is two-sided if and only if $S\mu = p\mathbb{Z}$ for some $p\in\mathbb{N}$.

We record the following results that will be used throughout the remainder of this paper.

\begin{prop}[{\cite[Lemmas 4.5 and 4.7]{ChoRuskuc2}}]
\label{prop:ChoRuskuc}
With the notation above, the following hold:
\begin{thmenumerate}
    \item \label{it:CR1} $\P = \bigsqcup_{a,b\in\mathbb{N}_0}\P_{a,b}$ is the disjoint union of its subsemigroups $\P_{a,b}$, each of which is isomorphic to $\mathbb{N}$;
    \item \label{it:CR2} any finitely generated subsemigroup of $\P$ has a non-empty intersection with only finitely many $\P_{a,b}$. \hfill\qed
\end{thmenumerate}
\end{prop}

A straightforward modification of \ref{it:CR1} holds for $\PE$: $\PE = \bigsqcup_{a,b\in\mathbb{N}_0}\PE_{a,b}$ is the disjoint union of its subsemigroups $\PE_{a,b}$, and each $\PE_{a,b}$ is isomorphic to $\mathbb{N}_0$ when $a+b>0$ and to $\mathbb{N}$ when $a=b=0$.

\begin{prop}[{\cite[Corollary 2.5]{Reilly}}]
\label{prop:Reilly}
Let $u\in\FI_1$ be non-idempotent. Then, the inverse subsemigroup $\langle u, u^{-1}\rangle$ of $\FI_1$ is isomorphic to $\FI_1$. \hfill\qed
\end{prop}

This is actually stated in more generality in \cite{Reilly}, with $u$ coming from an arbitrary free inverse semigroup.

\section{The monogenic free inverse semigroup is not Howson}
\label{sec:FI1NonHowson}

For the remainder of the paper we solely consider intersections of finitely generated subsemigroups of $\FI_1$. 
From the preliminary discussion above, we can quickly identify certain pairs of subsemigroups of $\FI_1$ whose intersections are finitely generated.

\begin{prop}
\label{prop:FI1fgintersectpairs}
Let $S$ and $T$ be finitely generated subsemigroups of $\FI_1$. If one of the following holds, then $S\cap T$ is finitely generated:
\begin{thmenumerate}
    \item \label{it:FI1fgintersect1} both $S$ and $T$ are inverse;
    \item \label{it:FI1fgintersect2} at least one of $S$ or $T$ is one-sided.
\end{thmenumerate}
\end{prop}
\begin{proof}
If \ref{it:FI1fgintersect1} holds, then the result follows from the inverse semigroup Howson property of $\FI_1$ (Theorem~\ref{thm:JonesTrotter}) and Proposition \ref{prop:invplaingen}.
Suppose \ref{it:FI1fgintersect2} holds. Without loss of generality, assume $S\subseteq \PE$ and $S = \langle X \rangle$ for some finite $X\subseteq \PE$. Note that $\P(S)$ is an ideal of $S$ and $E(S) = S\setminus\P(S) =\langle X\setminus\P(S)\rangle \leq \E$ is finite. Hence, $\P(S)$ is finitely generated by Proposition~\ref{prop:FiniteGenerationLarge}. By Proposition~\ref{prop:ChoRuskuc}, $\P(S)$ has a non-empty intersection with only finitely many $\P_{a,b}$, and hence so does $\P(S)\cap \P(T) = \P(S\cap T)$. Thus, $\P(S\cap T)$ is finitely generated by Propositions~\ref{prop:ChoRuskuc} and \ref{prop:SitSiu}. Since $(S\cap T)\setminus\P(S\cap T) \subseteq E(S)$ is finite, $S\cap T$ is finitely generated by Proposition~\ref{prop:FiniteGenerationLarge}.
\end{proof}

We now present an example demonstrating that $\FI_1$ does not have the semigroup Howson property. 
In fact, the two subsemigroups $S$ and $T$ in this example possess the following properties, which nicely complement the above proposition: both $S$ and $T$ are two-sided, one  of them is isomorphic to $\FI_1$, and their intersection is one-sided.

\begin{ex}
\label{ex:HowsonCounter1}
Let $S$ and $T$ be finitely generated subsemigroups of $\FI_1$ defined as
\[
S := \bigl\langle (-2, -2, 0), (0, 2, 2)\bigr\rangle, \quad\quad\; T := \bigl\langle (-2, -2, 1), (0, 2, 2)\bigr\rangle.
\]
Observe that both $S$ and $T$ are two-sided and that $S$ is an inverse semigroup. Indeed, $S$ is isomorphic to $\FI_1$ via the map defined by $(0, 2, 2) \mapsto (0, 1, 1)$, $(-2, -2, 0)\mapsto (-1, -1, 0)$ (Proposition~\ref{prop:Reilly}). It is easy to see that every element of $S$ can be uniquely written in the form $(-2a, 2x, 2b)$ where $a, x, b\in\mathbb{Z}$, $a,b\geq 0$ with $a+b>0$ and $-a\leq x\leq b$. In particular, $u\rho$ is even for every $u\in S$. On the other hand, we show that $u\rho$ is odd for every $u\in\NE(T)$. We proceed via a series of claims, with the desired result being Claim \ref{cl:HowsonCounter4}.

\begin{claim}
\label{cl:HowsonCounter1}
Let $u\in \N(T)$ be such that $u = x_1\dots x_n$ where each $x_i\in\{(-2, -2, 1), (0, 2, 2)\}$. Suppose that no proper initial subword of the product is an idempotent. Then, $u\rho = 1$.
\end{claim}

\begin{proof}
The assumptions imply that $(x_1
\dots x_i)\mu<0$ for all $i\in [n]$.
It follows that $x_1 = (-2, -2, 1)$. In particular, the claim holds when $n=1$.
For $n>1$ we have inductively $(x_1\dots x_{n-1})\rho = 1$. Now, either $x_n = (-2, -2, 1)$ or $x_n = (0, 2, 2)$. Since $T\mu = 2\mathbb{Z}$, if the former holds, we have 
\[
u\rho = ((x_1\dots x_{n-1})(-2, -2, 1))\rho = \max((x_1\dots x_{n-1})\rho,\, (x_1\dots x_{n-1})\mu + 1) = (x_1\dots x_{n-1})\rho = 1,
\]
while in the latter case
\[
u\rho = ((x_1\dots x_{n-1})(0, 2, 2))\rho = \max((x_1\dots x_{n-1})\rho,\, (x_1\dots x_{n-1})\mu + 2) = (x_1\dots x_{n-1})\rho = 1,
\]
completing the  proof.
\end{proof}

\begin{claim}
\label{cl:HowsonCounter2}
Let $u\in E(T)$ be such that $u = x_1\dots x_n$ where each $x_i\in\{(-2, -2, 1), (0, 2, 2)\}$. Suppose that no proper initial subword of the product is an idempotent. Then, $u\rho$ is odd.
\end{claim}

\begin{proof}
The assumption means that in all proper initial subwords the terminal vertex is on the same side in 
relation to the initial vertex. We split our considerations into two cases depending on the value of $x_1$.

\emph{Case 1: $x_1 = (-2, -2, 1)$.} Note that every proper initial subword must be negative and that $x_n = (0, 2, 2)$. The initial subword $x_1\dots x_{n-1}$ satisfies the assumptions from Claim~\ref{cl:HowsonCounter1}. Notice that $(x_1\dots x_{n-1})\mu = -2$, and then the final multiplication by $(0, 2, 2)$ does not increase the value of $\rho$. Hence, $u\rho = (x_1\dots x_n)\rho = (x_1\dots x_{n-1})\rho = 1$ by Claim~\ref{cl:HowsonCounter1}.

\emph{Case 2: $x_1 = (0, 2, 2)$.} In this case, every proper initial subword is positive and $x_n = (-2, -2, 1)$. Let 
\[
k := \min\{1\leq i \leq n: (x_1\dots x_i)\rho = u\rho\}.
\]
We claim that $x_k$ must be $(-2, -2, 1)$. Suppose otherwise. Then, 
\[
\bigl((x_1\dots x_{k-1})(0, 2, 2)\bigr)\rho = \max\bigl((x_1\dots x_{k-1})\rho, (x_1\dots x_{k-1})\mu + 2\bigr) = (x_1\dots x_{k-1})\mu + 2,
\]
by minimality of $k$; in particular, $(x_1\dots x_k)\rho = (x_1\dots x_k)\mu$. Since $x_k$ is not the final letter, $x_1\dots x_k$ is followed by $(0, 2, 2)$ or $(-2, -2 ,1)$. 
In each case we have:
\begin{align*}
    &\bigl((x_1\dots x_k)(0, 2, 2)\bigr)\rho = \max\bigl((x_1\dots x_k)\rho, (x_1\dots x_k)\mu + 2\bigr) = (x_1\dots x_k)\mu + 2 > (x_1\dots x_k)\rho;\\
    &\bigl((x_1\dots x_k)(-2, -2, 1)\bigr)\rho = \max\bigl((x_1\dots x_k)\rho, (x_1\dots x_k)\mu+1)\bigr) = (x_1\dots x_k)\mu+1>(x_1\dots x_k)\rho.
\end{align*}
Therefore
\[
(x_1\dots x_n)\rho\geq (x_1\dots x_{k+1})\rho> (x_1\dots x_k)\rho,
\]
which contradicts the choice of $k$, and thus establishes the claim that $x_k = (-2, -2, 1)$.  
The following is now straightforward:
\[
u\rho = \bigl((x_1\dots x_{k-1})(-2, -2, 1)\bigr)\rho = \max\bigl((x_1\dots x_{k-1})\rho, (x_1\dots x_{k-1})\mu+1\bigr) = (x_1\dots x_{k-1})\mu + 1,
\]
which is odd since $T\mu = 2\mathbb{Z}$. This completes the proof of Claim \ref{cl:HowsonCounter2}.
\end{proof}

\begin{claim}
\label{cl:HowsonCounter3}
For any $u\in E(T)$, $u\rho$ is odd.
\end{claim}
\begin{proof}
We may assume that $u$ is the product of idempotents $e_1, \dots, e_n\in E(T)$, each of which satisfies the assumption in Claim~\ref{cl:HowsonCounter2}. Then, $u\rho = \max(e_1\rho, \dots ,e_n\rho)$ is odd since each $e_i\rho$ is.
\end{proof}

\begin{claim}
\label{cl:HowsonCounter4}
For any $u\in\NE(T)$, $u\rho$ is odd.
\end{claim}

\begin{proof}
When $u\in E(T)$, the assertion is Claim \ref{cl:HowsonCounter3}.
Let us consider $u\in \N(T)$.
Aiming for a contradiction, assume
that $u\rho$ is even. Since $T\mu = 2\mathbb{Z}$, we may write $u$ as $(-a, -2x, 2b)$ where $a,b\in\mathbb{N}_0$ and $x\in\mathbb{N}$ with $-a\leq -2x$. But $(-a, 0, 2b) = u(0, 2, 2)^x\in E(T)$, contradicting Claim~\ref{cl:HowsonCounter3}. 
\end{proof}

We are now ready to prove that $S\cap T$ is not finitely generated.
Claim \ref{cl:HowsonCounter4} and the fact that $u\rho$ is even for all $u\in S$ imply that
$S\cap T\subseteq \P$.
On the other hand, for every $a\in \mathbb{N}_0$ we have 
\[
(-2a,2,2)=(-2,-2,0)^a (0,2,2)^{a+1}=(-2, -2, 1)^a(0, 2, 2)^{a+1}\in S\cap T.
\]
This means that the one-sided subsemigroup $S\cap T$ has a non-empty intersection with all $\P_{2a, 0}$,
and hence it is not finitely generated by Proposition~\ref{prop:ChoRuskuc}\ref{it:CR2}.
\end{ex}

As indicated in the Introduction, we have now established the first main result of this paper, Theorem \ref{thmA:Howson}.

We conclude this section by characterising those subsemigroups of $\FI_1$ with the semigroup Howson property.

\begin{prop}
\label{prop:subFI1iff}
Let $S$ be a subsemigroup of $\FI_1$. Then, the following are equivalent:
\begin{thmenumerate}
    \item \label{it:subFI11} $S$ has the semigroup Howson property;
    \item \label{it:subFI12} $S$ is one-sided;
    \item \label{it:subFI13} $S$ does not contain a copy of $\FI_1$.
\end{thmenumerate}
\end{prop}
\begin{proof}
\ref{it:subFI11} $\Rightarrow$ \ref{it:subFI13} follows from $\FI_1$ not being Howson and Proposition~\ref{prop:subHowson}. 

\noindent \ref{it:subFI13} $\Rightarrow$ \ref{it:subFI12} Suppose $S$ is two-sided. Then, $S$ contains a positive element $u_1 = (-a_1, x_1, b_1)$ and a negative element $u_2 = (-a_2, -x_2, b_2)$ where $x_1, x_2>0$. It is routine to check that $u_1^{x_2}u_2^{x_1}u_1^{x_2}$ and $u_2^{x_1}u_1^{x_2}u_2^{x_1}$ are inverses of each other. By Proposition~\ref{prop:Reilly}, $S$ contains a copy of $\FI_1$.

\noindent\ref{it:subFI12} $\Rightarrow$ \ref{it:subFI11}. If $S$ is one-sided, so is any subsemigroup of $S$. The result follows from Proposition~\ref{prop:FI1fgintersectpairs}~\ref{it:FI1fgintersect2}.
\end{proof}

Of course, the corresponding classification for the non-free monogenic inverse semigroups is vacuous: they are all Howson, and hence so are all their subsemigroups.

\section{Finitely generated subsemigroups of the monogenic free inverse semigroup}
\label{sec:Character}

Motivated by the main 
finding from the previous section, 
namely that the monogenic free inverse semigroup $\FI_1$ is not Howson,
we proceed to further investigate  finitely generated subsemigroups of  $\FI_1$.
Our main goal is to establish a positive answer to the FGIP for $\FI_1$. 
To do so, in this section we provide a characterisation of finitely generated subsemigroups of $\FI_1$. We have already observed differences between one-sided and two-sided subsemigroups in the previous section, and, unsurprisingly, there will be two different characterisations, one for each. The characterisation in the two-sided case is much more complicated than that in the one-sided case; in fact, we already have all the ingredients for the 
latter, and we present it in Subsection~\ref{subsec:1side}. For the two-sided case, we first obtain auxiliary results in Subsections~\ref{subsec:PIG} and \ref{subsec:Periodic}, and the characterisation is established in Subsection~\ref{subsec:2side}.

\subsection{One-sided subsemigroups}
\label{subsec:1side}

\begin{thm}
\label{thm:1side}
Let $S\subseteq \PE$ (resp.\ $S\subseteq \NE$) be a one-sided subsemigroup of $\FI_1$. Then, $S$ is finitely generated if and only if it has a non-empty intersection with only finitely many $\PE_{a,b}$ (resp.\ $\NE_{a,b}$.)
\end{thm}

\begin{proof}
Without loss of generality, assume $S\subseteq \PE$.

\noindent ($\Rightarrow$) Suppose $S = \langle X\rangle$ where $X\subseteq \PE$ is finite. Since $\P(S)$ is an ideal of $S$, $E(S) = S\setminus\P(S) = \langle X\setminus\P(S)\rangle \subseteq \E$ is finite. Thus, $E(S)$ has a non-empty intersection with only finitely many $\PE_{a,b}$. By Proposition~\ref{prop:FiniteGenerationLarge}, $\P(S)$ is finitely generated. Hence, $\P(S)$ has a non-empty intersection with only finitely many $\PE_{a,b}$ by Proposition~\ref{prop:ChoRuskuc} \ref{it:CR2}.

\noindent ($\Leftarrow$) Each $\PE_{a,b}$ is isomorphic to $\mathbb{N}_0$. Thus, each non-empty $\PE_{a,b}\cap S$ is finitely generated by Proposition~\ref{prop:SitSiu}. Then, $S$ is the finite union of finitely generated subsemigroups and so finitely generated.
\end{proof}

Based on the above proof, we view a finitely generated one-sided subsemigroup of $\FI_1$ as a finite union of subsemigroups of $\mathbb{N}_0$.

\subsection{Purely idempotent-generated elements}
\label{subsec:PIG}

This concept was introduced by Jones and Trotter \cite{JonesTrotter}. Here we establish
a few auxiliary results needed for our characterisation in the two-sided case.

\begin{defn}
\label{defn:PIG}
Let $S$ be a semigroup. An element $e\in S$ is said to be \emph{purely idempotent-generated} if the following holds: whenever $e$ is written as $s_1\dots s_n$ where each $s_i\in S$, every $s_i$ is an idempotent. The set of purely idempotent-generated elements of $S$ is denoted by $\PIG(S)$. The complement of $\PIG(S)$ in $S$ is denoted by $S^\circ$. 
\end{defn}

Observe that each of $\PIG(S)$ and $S^\circ$ can be empty. If $S^\circ$ is not empty, it is easy to see that it is an ideal of $S$. It is clear that every purely idempotent-generated element is an idempotent, i.e. $\PIG(S)\subseteq E(S)$.

The following result from \cite{JonesTrotter} was crucial in the proof that $\FI_1$ has the inverse semigroup Howson property.

\begin{prop}[{\cite[Proposition 1.5]{JonesTrotter}}]
\label{prop:JonesTrotterPIG}
Let $S$ be an inverse subsemigroup of $\FI_1$. Then, $S$ is finitely generated if and only if $\PIG(S)$ is finite. \qed
\end{prop}

For two-sided subsemigroups of $\FI_1$, we will often consider their maximal inverse subsemigroups. Given a subsemigroup $T$ of an inverse semigroup $S$, we write $\inv(T) := T\cap T^{-1}$. Observe that $\inv(T)$ may be empty; otherwise, it is the unique maximal inverse subsemigroup of $T$.

We also need the concept of $E$-unitarity.
An inverse semigroup $S$
is \emph{$E$-unitary} if the following holds: if $e\in E(S)$ and $s\in S$
are such that $es\in E(S)$ then necessarily $s\in E(S)$. It is easy to check that if $S$ is $E$-unitary, then $se\in E(S)$ also implies $s\in E(S)$. 
Notice that any inverse subsemigroup of an $E$-unitary inverse semigroup is again $E$-unitary. It is a standard fact that free inverse semigroups (of any rank) are $E$-unitary (see, for example, \cite[Chapter 6]{Lawson}).

The following result extends \cite[Proposition 1.1]{JonesTrotter} to plain subsemigroups of $E$-unitary inverse semigroups.

\begin{lem}
\label{lem:EuniPIGGreen}
Let $S$ be an $E$-unitary inverse semigroup and let $T$ be a subsemigroup of $S$.
Then, $e\in E(T)$ is not purely idempotent-generated if and only if there 
exists a non-idempotent $s\in \inv(T)$ such that $e\mathscr{R}s$.
\end{lem}
\begin{proof}
The backward direction follows since $e = ss^{-1}$.
For the forward direction, suppose $e\in E(T)$ is not purely idempotent-generated and write $e = t_1\dots t_n$ where each $t_i\in T$ and at least one of $t_i$ is not an idempotent. Let $t_i$ be the first non-idempotent element in the product. Write $u = t_1\dots t_i$ and $v = t_{i+1}\dots t_n$, so that $e = uv$. Note that all of $u$, $v$, $eu$, $ve$ are non-idempotents by $E$-unitarity of $S$. Moreover, $eu = (eu)(ve)(eu)$ and $ve = (ve)(eu)(ve)$. This means that $eu$ and $ve$ are inverses of each other. Therefore, $eu \mathscr{R} e = (eu)(ve) = (eu)(eu)^{-1}$, 
completing the proof.
\end{proof}

\begin{lem}
\label{lem:EuniPlainSamePIG}
Let $S$ be an $E$-unitary inverse semigroup and let $T$ be a subsemigroup of $S$. Then, $\PIG(T) = \PIG(\inv(T))$.
\end{lem}
\begin{proof}
Since $\inv(T)\subseteq T$ and $\PIG(T)\subseteq E(T)\subseteq \inv(T)$, we have that $e\in \PIG(T)$ implies $e\in \PIG(\inv(T))$. For the reverse inclusion, let $e\in E(T)\setminus\PIG(T)$. Lemma~\ref{lem:EuniPIGGreen} implies that there is a non-idempotent $s\in \inv(T)$ such that $e\mathscr{R}s$; in particular, $e = ss^{-1}$. Thus, $e\notin\PIG(\inv(T)))$. This establishes the reverse inclusion and completes the proof.
\end{proof}

\begin{rmk}
\label{rmk:NonEuni}
The two lemmas above do not hold if the ambient inverse semigroup $S$ is not $E$-unitary. Let $S$ be the symmetric inverse semigroup $\mathcal{I}_2$ on $\{1,2\}$. Clearly, $S$ is not $E$-unitary. 
Take $T = \{\id, 0,f,g\}\leq \mathcal{I}_2$, where $\id$ is the identity transformation, $0$ is the empty transformation,
$f = \bigl(\begin{smallmatrix}
		1 & 2 \\
		2 & -
	\end{smallmatrix}\bigr)$,\,
$g  = \bigl(\begin{smallmatrix}
		1 & 2\\
		- & 2
	\end{smallmatrix}\bigr)$.
Notice that $\inv(T) = \{\id, g, 0\}$ and $E(T) = \inv(T) = \PIG(\inv(T))$. Here, $0\in \PIG(\inv(T))\setminus \PIG(T)$.
\end{rmk}

\begin{lem}
\label{lem:subFI1PIG}
Let $S$ be a finitely generated subsemigroup of $\FI_1$. Then, $\PIG(S)$ is finite and $\inv(S)$ is finitely generated.
\end{lem}
\begin{proof}
For the first statement, recall that $S^\circ$ is either empty or an ideal of $S$.
Thus, we have $\PIG(S) \subseteq \langle X\setminus S^\circ\rangle \subseteq  E(S)$ and so $\PIG(S)$ is finite.
Note that this argument does not use the assumption that $S\leq \FI_1$.
The second statement now follows from the fact that $\FI_1$ is $E$-unitary, 
Lemma~\ref{lem:EuniPlainSamePIG} and Proposition~\ref{prop:JonesTrotterPIG}.
\end{proof}

\subsection{Periodic subsets}
\label{subsec:Periodic}

In this subsection, we introduce the notion of periodic subsets of $\FI_1$. It is closely related to some considerations in Silva's paper \cite{SilvaRationalSubset} regarding rational subsets of the monogenic free inverse monoid, which we will review and use in the next section. At that point we will dwell in more detail on the salient relationships between them and what we are about to present here (Remark \ref{rem:salient}).

\begin{defn}
\label{defn:PeriodicSets}
Let $L\subseteq \PE$ and let $n,\nu$ be natural numbers. We say that $L$ is \emph{horizontally periodic from $n$ with period $\nu$} if it satisfies the following: for all $a,b\in\mathbb{N}_0$ with $a\geq n$, we have 
\[
(\PE_{a,b}\cap L)\mu = (\PE_{a+\nu, b}\cap L)\mu. 
\]
Similarly, we say that $L$ is \emph{vertically periodic from $n$ with period $\nu$} if it satisfies the following: for all $a,b\in\mathbb{N}_0$ with $b\geq n$, we have 
\[
(\PE_{a,b}\cap L)\mu = (\PE_{a,b+\nu}\cap L)\mu.
\]
We say that $L$ is \emph{periodic from $n$ with period $\nu$} if it is both horizontally and vertically periodic from $n$ with period $\nu$. We also say that $L$ is \emph{$\nu$-periodic from} $n$.
\end{defn}
\begin{rmk}
\label{rmk:PeriodicSets}
\begin{rmkenumerate}
    \item The terms `horizontal' and `vertical' stem from the following viewpoint.
Recall that $\PE = \bigcup_{a,b\in\mathbb{N}_0}\PE_{a,b}$ and that each $\PE_{a,b}\cong\mathbb{N}_0$ if $a+b>0$ and $\PE_{0,0}\cong\mathbb{N}$. Thus, we may represent $\PE$ as the grid $\mathbb{N}_0\times\mathbb{N}_0$
 with each point $(a,b)$ corresponding to $\PE_{a,b}$. 
For a subset $L\subseteq \PE$, view it as the disjoint union of the sets $L\cap \PE_{a,b}$, and then further interpret each of those sets as a subset of $\mathbb{N}_0$ via $\mu$.
If $L$ is, say, vertically periodic, then these subsets of $\mathbb{N}_0$ periodically repeat along every vertical line $\{(a,b) : b\in\mathbb{N}_0\}$. 
    \item We have a simpler description of periodic behaviour when $L\subseteq \PE$ consists solely of idempotents. In this case, $L$ is horizontally periodic from $n$ with period $\nu$ if the following holds: for all $a,b\in\mathbb{N}_0$ with $a\geq n$, $(-a, 0, b)\in L \Leftrightarrow (-(a+\nu), 0, b)\in L$. The description for vertical periodicity is analogous.
   
    \item If $L\subseteq\PE$ is horizontally (resp.\ vertically) periodic from $n$ with period $\nu$, then $L$ is also horizontally (resp.\ vertically) periodic from $n'$ with period $m\nu$ where $n'\geq n$ and $m\in\mathbb{N}$.
    \item If $L$ is vertically periodic from $n_1$ with period $\nu_1$ and is horizontally periodic from $n_2$ with period $\nu_2$, then $L$ is periodic from $n$ with period $\nu$ where $n\geq \max(n_1, n_2)$ and $\nu$ is any common multiple of $\nu_1$ and $\nu_2$.
\end{rmkenumerate}
\end{rmk}

There is an obvious dual notion of periodicity for subsets of $\NE$, but we will not require this. Instead, if $L\subseteq \NE$, then we revert our considerations to $L^{-1}\subseteq \PE$.

For a two-sided subsemigroup $S$ of $\FI_1$, we will consider periodic behaviour of $\PE(S)$ and $\PE(S^{-1})$. In this case, it will be more convenient to have the same periodicity and its starting point for both of them. Throughout the paper, we say that a two-sided $S$ is \emph{periodic from $n$ with period $\nu$} if both $\PE(S)$ and $\PE(S^{-1})$ are $\nu$-periodic from $n$; we sometimes say that $S$ is \emph{$\nu$-periodic from $n$}. From the remark above, note that if $\PE(S)$ (resp. $\PE(S^{-1})$) is periodic from $n_1$ (resp. $n_2$) with period $\nu_1$ (resp. $\nu_2$), then we can always choose $n\geq \max(n_1, n_2)$ and $\nu$ as a common multiple of $\nu_1$ and $\nu_2$ so that both $\PE(S)$ and $\PE(S^{-1})$ are $\nu$-periodic from $n$.

We define two special numbers associated with periodic subsets.

\begin{defn}
\label{defn:HorizonVertNumbers}
Let $L\subseteq\PE$ be periodic from $n$ with period $\nu$, and let
\[
U_H:= \bigcup \{ L\cap \P_{a,b}\colon a\geq n,\ b\geq 0\}
\quad\text{and}\quad
U_V:= \bigcup \{ L\cap \P_{a,b}\colon a\geq 0,\ b\geq n\}.
\]
Define:
\begin{alignat*}{4}
h_{L} &:=
\min\bigl\{x+d: (-c, x, x+d) \in U_H\bigr\}\ && \text{if } U_H\neq \emptyset,\quad&&
h_L:=0\ &&\text{if } U_H=\emptyset;
\\
v_{L} &:=
\min\bigl\{c+x: (-c, x, x+d)\in U_V\bigr\}  &&\text{if } U_V\neq \emptyset,&&
v_L:=0&& \text{if } U_V=\emptyset.
\end{alignat*}
\end{defn}

We will be concerned with periodic two-sided subsemigroups $S\leq\FI_1$ in the next subsection. 
In that case we will need the $h$- and $v$-parameters for each of the one-sided restrictions $\PE(S)$ and $\NE(S)$. We want to establish that these are all non-zero. First, however, we prove the following general fact, which does not assume periodicity:

\begin{lem}
\label{lem:SInvPhi}
Let $S\leq \FI_1$ be two-sided. Then $S\mu = (\inv(S))\mu = p\mathbb{Z}$ for some $p\in \mathbb{N}$.
\end{lem}

\begin{proof}
Since $S$ is two-sided, $S\mu = p\mathbb{Z}$ for some $p\in\mathbb{N}$. Thus, there are $u_1 = (-a_1, p, b_1)$ and $u_2 = (-a_2, -p, b_2)$ in $S$. It is a routine matter to check that $u_1u_2u_1$ and $u_2u_1u_2$ are inverses of each other, with the images under $\mu$ being $p$ and $-p$, respectively.
\end{proof}

\begin{lem}
\label{lem:NonZeroNums}
Let $S \leq \FI_1$ be two-sided. Suppose that $S$ is periodic from $n$ with period $\nu$. Then, none of $h_{\PE(S)}$, $h_{\PE(S^{-1})}$, $v_{\PE(S)}$, $v_{\PE(S^{-1})}$ are zero.
\end{lem}

\begin{proof}
We show $v_{\PE(S)}$ is non-zero. Since $\inv(S)$ is not a semilattice, Lemma~\ref{lem:SInvPhi} implies that
there is a non-idempotent $u\in \inv(S)$; we assume $u$ is positive without loss of generality. Observe that for $i\in\mathbb{N}$ the element $u^{i}u^{-i}u$ is positive and the value $(u^{i}u^{-i}u)\rho-(u^{i}u^{-i}u)\mu$ increases with~$i$. In particular, $u^{i}u^{-i}u\in \P_{a,b}$ where $a = (u^{i}u^{-i}u)\lambda$ and $b = (u^{i}u^{-i}u)\rho-(u^{i}u^{-i}u)\mu$. Thus, $\bigcup_{a\geq 0, b\geq n}(\PE(S)\cap \P_{a,b})$ is non-empty and so $v_{\PE(S)}\neq 0$.
The argument for $h_{\PE(S)}$ is analogous and for $v_{\PE(S^{-1})}$ and $h_{\PE(S^{-1})}$ dual.
\end{proof}

\subsection{Two-sided subsemigroups}
\label{subsec:2side}
In this subsection, we establish the characterisation for finitely generated two-sided subsemigroups of $\FI_1$. We begin this by fixing notation and introducing an auxiliary result from \cite{OliveiraSilva}.

Throughout this subsection, unless stated otherwise, $S$ will stand for a two-sided subsemigroup of $\FI_1$. Further, $p$ will denote the unique natural number such that $S\mu = (\inv(S))\mu = p\mathbb{Z}$ (Lemma~\ref{lem:SInvPhi}). 

Motivated by \cite{OliveiraSilva}, we define:
\begin{align*}
    a_{S}^\ast := \min\bigl\{a\in\mathbb{N}_0: (-a, qp, b)\in \inv(S)^\circ\bigr\};\quad\;
    b_{S}^\ast := \min\bigl\{b\in\mathbb{N}_0: (-a, qp, b)\in \inv(S)^\circ\bigr\}.
\end{align*}
Note that $\inv(S)^\circ := \inv(S)\setminus\PIG(\inv(S)) = \inv(S)\setminus\PIG(S)$ by Lemma~\ref{lem:EuniPlainSamePIG} and $E$-unitarity of $\FI_1$. Thus, $a_S^\ast$ (resp.\ $b_S^\ast$) is the least $\lambda$ value (resp.\ $\rho$ value) of triples in $\inv(S)\setminus\PIG(S)$. It is also useful to determine $a_S^\ast$ and $b_S^\ast$ from idempotents. For each triple $u=(-a, qp, b)\in \inv(S)^\circ$, we have an idempotent $uu^{-1} = (-a, 0, b)\in E(\inv(S)^\circ)$. Observe from Lemma~\ref{lem:EuniPIGGreen} that $E(\inv(S)^\circ)= E(\inv(S))\setminus\PIG(S) = E(S)\setminus\PIG(S)$. Hence, $a_S^\ast$ (resp.\ $b_S^\ast$) is the least $\lambda$ value (resp.\ $\rho$ value) of triples in $E(S)\setminus\PIG(S)$.

\begin{prop}[{\cite[Lemma 2.2]{OliveiraSilva}}]
\label{prop:OliveiraSilva}
Let $S$ be as above. Then, there are $\alpha_S = (-a_{S}^\ast, p, p+b_\alpha)$ and $\beta_S = (-(a_\beta+p), -p, b_S^\ast)$ in $\inv(S)$ for some $a_\beta, b_\alpha\in\mathbb{N}_0$. \qed
\end{prop}

\begin{rmk}
The definitions of $a_S^\ast$ and $b_S^\ast$ and the statement above are slightly different in \cite{OliveiraSilva}, as the focus there is on \emph{finitely generated inverse} subsemigroups of $\FI_1$. Nonetheless, this does not affect the result and the same proof applies.
\end{rmk}

For the remainder of this section we consider the parameters $a_{S}^\ast, b_S^\ast, a_\beta, b_\alpha$ and the elements $\alpha_S$ and $\beta_S$ to be fixed. When the context is clear, for example if there is a single two-sided subsemigroup under consideration, we omit the subscript $S$ in this notation.

We now have all the ingredients to state the theorem characterising finitely generated two-sided subsemigroups:

\begin{thm}
\label{thm:2side}
Let $S$ be a two-sided subsemigroup of the monogenic free inverse semigroup $\FI_1$. Then, $S$ is finitely generated if and only if the following hold:
\begin{thmenumerate}
    \item \label{it:2side1} $\inv(S)$ is finitely generated;
    \item \label{it:2side2} $S$ is periodic from $n$ with period $\nu$ for some $n,\nu\in\mathbb{N}$;
    \item \label{it:2side3} $a^\ast \leq \min(v_{\PE(S)}, v_{\PE(S^{-1})})$ and $b^\ast \leq \min(h_{\PE(S)}, h_{\PE(S^{-1})})$.
\end{thmenumerate}
\end{thm}

Further lemmas are needed before we begin the proof.

\begin{lem}
\label{lem:PeriodicSemilattice}
With the assumption above, if $\PIG(S)$ is finite, then $E(S)$ is $p$-periodic from some~$n_{E}$.
\end{lem}

\begin{proof}
Define 
\[
m := \max\bigl(\bigl\{a^\ast+p+b_{\alpha}, a_{\beta}+p+b^\ast\bigr\}\cup \bigl\{a+b : (-a, 0, b)\in\PIG(S)\bigr\}\bigr).
\]
Recall that the set $I_{2m} := \bigcup_{i\geq 2m}D_i$ is an ideal of $\FI_1$.
 We claim that if $e = (-a_e, 0, b_e)\in I_{2m}\cap E(S)$ then both $(-(a_e+p), 0, b_e)$ and $(-a_e, 0, b_e+p)$ are in $E(S)$. To see this, note that either $a_e\geq m$ or $b_e\geq m$. 
In the former case routine calculations show 
$(e\beta e)(e\beta e)^{-1} = (-(a_e+p),0,b_e)\in E(S)$ and $(e\beta e)^{-1}(e\beta e) = (-a_e, 0, b_e+p)\in E(S)$. 
The latter case is analogous, using $e\alpha e\in \inv(S)$, proving the claim.

Note that each idempotent in $I_{2m}\cap E(S)$ can be written as $(-(a^\ast +xp +r),0, b^\ast+yp+s)$ where $x,y\in\mathbb{N}_0$ and $r,s\in[0,p-1]$, due to our definitions of $a^\ast$ and $b^\ast$. For each $(r,s)\in[0,p-1]^2$, we define 
\[
E_{r,s} := \bigl\{(-a, 0, b)\in I_{2m}\cap E(S) : a = a^\ast +xp+r \text{ and } b = b^\ast+yp+s \text{ for some } x,y\in\mathbb{N}_0\bigr\}.
\]
So, $I_{2m}\cap E(S) = \bigcup E_{r,s}$, where the union is taken over all $(r,s)\in[0,p-1]\times [0,r-1]$. 
Let $M_{r,s}$ be the set of maximal idempotents in $E_{r,s}$, under the partial order introduced in \eqref{eq:pord}.
Recall that $\E$ is order-isomorphic to $\mathbb{N}_0\times\mathbb{N}_0\setminus\{(0,0)\}$
under $\geq$. Hence, each $M_{r,s}$ is finite by Dickson's lemma; see for example
\cite[Corollary 6.12]{Diekert16}.

Define 
\[
n_E := \max \bigl\{e\lambda + e\rho : e\in M_{r,s},\ (r,s)\in [0,p-1]\times [0,p-1]\bigr\}.
\]
We prove that $E(S)$ is $p$-periodic from $n_{E}$. From the claim above, we already know that if $(-a, 0, b)\in E(S)$ with $a\geq n_{E}$ or $b\geq n_E$, then $(-(a+p), 0, b), (-a, 0, b+p)\in E(S)$.

We now show that if $f = (-(a+p), 0, b)\in E(S)$ with $a\geq n_E$ then $(-a, 0, b)\in E(S)$. Note that $f \in E_{r,s}$ for some $(r,s)\in [0, p-1]\times [0,p-1]$. So, write 
\[
f = (-(a+p), 0, b) = \bigl(-(a^\ast + xp +r), 0, b^\ast +yp+s\bigr)\in E_{r,s}
\]
for some $x, y\in \mathbb{N}_0$. Since $a+p+b > a \geq n_E$, there exists a maximal idempotent 
\[
f_m = \bigl(-(a^\ast +zp+r), 0, b^\ast + tp +s \bigr)\in M_{r,s}
\]
where $z, t\in \mathbb{N}_0$ and $f_m > f$; in particular, we have $z < x$ and $t\leq y$. Then, the claim implies 
\[
\bigl(-(a^\ast + (z+i)p + r), 0, b^\ast + (t+j)p + s\bigr)\in E(S)
\]
for every $i, j\in\mathbb{N}_0$. Thus, $(-a, 0, b)\in E(S)$. This shows that $E(S)$ is horizontally periodic from $n_E$ with period $p$. By symmetry, $E(S)$ is vertically periodic from $n_E$ with period $p$, and the lemma is proved.
\end{proof}

The next lemma asserts that $\P_{a,b}\cap S \subseteq \inv(S)$
for large enough $a,b\in\mathbb{N}_0$. Of course, there is a dual version in terms of $\N_{a,b}\cap S$.

\begin{lem}
\label{lem:LargeabInv}
Suppose $u\in S$ is such that $u\in \P_{a,b}$ with $a\geq a_\beta$ and $b\geq b^\ast$. Then, $u\in\inv(S)$.
\end{lem}

\begin{proof}
Write $u = (-a, qp, qp+b)\in \P(S)$ with $a\geq a_{\beta}$ and $b\geq b^\ast$. It is then routine to check that $u^{-1} = (-(a+qp), -qp, b) = \beta^qu\beta^q\in S$. 
\end{proof}

We now establish the main theorem:

\begin{proof}[Proof of Theorem~\ref{thm:2side}]
\noindent ($\Rightarrow$) Suppose that the two-sided subsemigroup $S\leq \FI_1$ is generated by a finite set $X$. By Lemma~\ref{lem:subFI1PIG}, $\PIG(S)$ is finite and $\inv(S)$ is finitely generated; this shows \ref{it:2side1}.

For \ref{it:2side2}, we show $S$ is $p$-periodic from some $n$. First, note that since $\PIG(S)$ is finite, there is some $n'$ such that $\PIG(S)\subseteq \bigcup_{0\leq a,b\leq n'}\PE_{a,b}$. Also, note that $\langle \PE(X)\rangle$ and $\langle\PE(X^{-1})\rangle$ are finitely generated one-sided subsemigroups. Using Theorem \ref{thm:1side}, choose $n''$ to be such that
\[
\langle\PE(X)\rangle\cup \langle\PE(X^{-1})\rangle\subseteq \bigcup_{a,b\in [0,n'']}\PE_{a,b}.
\]

Let $n_1\in\mathbb{N}$ be larger than any of the following:
\begin{bulletmidsentence}
    \item $n'$ and $n''$;
    \item $n_E$ where the $p$-periodicity of $E(S)$ starts (Lemma~\ref{lem:PeriodicSemilattice});
    \item $a^\ast + p +b_{\alpha}$ and $a_{\beta} + p + b^\ast$;
    \item $x\lambda + x\rho$ for all $x\in X$.
\end{bulletmidsentence}

To show $p$-periodicity of $\PE(S)$ we proceed via a series of claims.

\begin{claim}
\label{cl:2side1}
For $a,b\geq n_1$, we have 
\[
\textup{(1)}\ \bigl(\PE_{a,b}\cap \PE(S)\bigr)\mu = \bigl(\PE_{a+p, b}\cap \PE(S)\bigr)\mu; \text{ and}\quad
\textup{(2)}\ \bigl(\PE_{a,b}\cap \PE(S)\bigr) \mu = \bigl(\PE_{a,b+p}\cap \PE(S)\bigr)\mu .
\]
\end{claim}

\begin{proof}[Proof of Claim~\ref{cl:2side1}]
We show (2), and (1) is analogous. Since $E(S)$ is $p$-periodic from $n_E$ and ${n_1\geq n_E}$, we only need to consider the elements in $\P(S)$. For ($\subseteq$), assume $u=(-a, qp, qp+b)\in \P(S)$. By Lemma~\ref{lem:LargeabInv}, $u\in \inv(S)$. Thus, $uu^{-1} = (-a, 0, qp+b)\in E(S)$. Since $E(S)$ is $p$-periodic and $a,b\geq n_E$, we have $(-(a+ip), 0, b+jp)\in E(S)$ for all $i,j\in\mathbb{N}_0$; in particular, $(-a, 0, b+p)\in E(S)$. Then, $(-a, qp, qp+b+p) = u(-a, 0, b+p)\in \P(S)$.

For ($\supseteq$), assume $v = (-a, qp, qp+b+p)\in \P(S)$. By the same argument as above, there exists $e = (-a, 0, b)\in E(S)$. It is then routine to check that $(-a, qp, qp+b) = e\alpha^qe \in \P(S)$, as required.
\end{proof}

\begin{claim}
\label{cl:2side2}
Let $u = (-a, qp, qp+b)\in \P(S)$. Then, the following hold:
\begin{enumerate}[topsep=0mm, leftmargin=14mm,itemsep=1mm,label=\textup{(\arabic*)}]
    \item \label{it:cl2side2-1}if $a\geq n_1$, there is $e\in E(S)\setminus\PIG(S)$ such that $e\lambda = a+p$ and $e\rho \leq qp+b$;
    \item \label{it:cl2side2-2}
    if $b\geq n_1$, there is $f\in E(S)\setminus\PIG(S)$ such that $f\lambda \leq a+qp$ and $f\rho = b+p$.
\end{enumerate}
\end{claim}

\begin{proof}[Proof of Claim~\ref{cl:2side2}]
\ref{it:cl2side2-1} Suppose $a\geq n_1$ and write $u = x_1\dots x_m$ where each $x_i\in X$. Note that at least one of $x_i$ is a negative element since $a\geq n_1\geq n''$. Let $x_i = (-(c_i+r_ip), -r_ip, d_i)$ be the first negative generator that occurs in the product. Consider the idempotent 
\begin{align*}
    e' = x_i^{q}u^{r_i} &= \bigl(-(c_i+r_iqp),\, -r_iqp,\, d_i\bigr) \big(-a,\, r_iqp,\, r_iqp+b\bigr)\\
    &= \big(-\max(c_i+r_iqp, a+r_iqp),\, 0,\, \max(d_i, b)\bigr)\\
    &= \big(-(a+r_iqp),\, 0,\, \max(d_i,b)\bigr)\in E(S) &&\text{(since $a\geq n_1$)}.
\end{align*}
Note that 
\begin{align*}
    qp+b = u\rho &\geq (x_1\dots x_{i})\rho\\
    &= \max\bigl((x_1\dots x_{i-1})\rho, \, (x_1\dots x_{i-1}\bigr)\mu + x_{i}\rho)\\
    &\geq (x_1\dots x_{i-1})\mu+x_i\rho\\
    &\geq x_i\rho = d_i.
\end{align*}
Hence, $e'\rho\leq qp+b$. Since $E(S)$ is $p$-periodic from $n_E$ and $a\geq n_1\geq n_E$, there exists $e = (-(a+p),\, 0,\, e'\rho)\in E(S)$; it is clear that $e\notin \PIG(S)$.

\ref{it:cl2side2-2} Suppose $b\geq n_1$ and write $u = x_1\dots x_m$ where each $x_i\in X$. This time, let $x_j = (-(c_j+r_jp), \, -r_jp,\, d_j)$ be the last negative generator that occurs in the product. Consider the idempotent 
\[
f' = u^{r_j}x_j^{q} = (-\max(a,c_j),\, 0,\, r_jqp+b)\in E(S).
\]
Note that $u\mu = (x_1\dots x_{j-1})\mu - r_jp + (x_{j+1}\dots x_m)\mu$. Also, 
\begin{align*}
    a = u\lambda &\geq (x_1\dots x_{j})\lambda \\
    &=\max\bigl((x_1\dots x_{j-1})\lambda,\, x_j\lambda - (x_1\dots x_{j-1})\mu\bigr)\\
    &\geq c_j+r_jp-(x_1\dots x_{j-1})\mu \\
    &= c_{j} - qp + (x_{j+1}\dots x_m)\mu\\
    &\geq c_j - qp.
\end{align*}
Thus, $f'\lambda \leq a+qp$. By the $p$-periodicity of $E(S)$ and $b\geq n_E$, we have $f = (-f'\lambda, \, 0, \, b+p)\in E(S)\setminus\PIG(S)$, completing the proof of the claim.
\end{proof}

\begin{claim}
\label{cl:2side3}
The following hold:
\begin{enumerate}[topsep=0mm, leftmargin=14mm,itemsep=1mm,label=\textup{(\arabic*)}]
\item \label{it:cl2side3-1} if $a\geq n_1$, $(\PE_{a,b}\cap \PE(S))\mu \subseteq (\PE_{a+p, b}\cap \PE(S))\mu$;
\item \label{it:cl2side3-2} if $b\geq n_1$, $(\PE_{a,b}\cap \PE(S))\mu \subseteq (\PE_{a, b+p}\cap \PE(S))\mu$.
\end{enumerate}
\end{claim}

\begin{proof}[Proof of Claim~\ref{cl:2side3}]
We only need to consider positive elements since $n_1\geq n_E$. For \ref{it:cl2side3-1}, given $u = (-a, qp, qp+b)\in \P(S)$ with $a\geq n_1$, there is an idempotent $e = (-(a+p), 0, e\rho)$ with $e\rho\leq qp+b$ by \ref{it:cl2side2-1} of Claim~\ref{cl:2side2}. Then, $eu = (-(a+p), qp, qp+b)\in \P(S)$; this shows \ref{it:cl2side3-1}. The proof of \ref{it:cl2side3-2} is similar.
\end{proof}

We now show vertical periodicity of $\PE(S)$. Claim~\ref{cl:2side1} shows the vertical periodicity with period $p$ for $a,b\geq n_1$. Suppose $0\leq a <n_1$. Consider $b\geq n_1$ such that $\PE_{a,b}\cap \PE(S)\neq\emptyset$; choose the smallest such $b$. By \ref{it:cl2side3-2} of Claim~\ref{cl:2side3}, we have 
\begin{align}
\bigl(\PE_{a,b}\cap \PE(S)\bigr)\mu \subseteq \bigl(\PE_{a,b+p}\cap \PE(S)\bigr)\mu \subseteq \bigl(\PE_{a,b+2p}\cap \PE(S)\bigr)\mu\subseteq\cdots .\label{eq:seq}    
\end{align}
Note that each subsemigroup in the sequence (\ref{eq:seq}) is a subsemigroup in $\mathbb{N}_0$; in particular, the sequence stabilises. Thus, for each $0\leq a<n_1$, 
there is $n_a\geq n_1$ such that 
$(\PE_{a, b}\cap \PE(S))\mu = (\PE_{a, b+p} \cap \PE(S))\mu$ for all $b\geq n_a$.
Define $n := \max\{n_a: 0\leq a <n_1\}$. Then, $\PE(S)$ is vertically periodic from $n$ with period $p$. An analogous argument shows the horizontal periodicity of $\PE(S)$ with period $p$. The periodicity of $\PE(S^{-1})$ is dual. This concludes the proof that $S$ satisfies \ref{it:2side2} from the statement of Theorem \ref{thm:2side}.

We now show that \ref{it:2side3} holds. 
We will retain the numbers $n_1$ from the proof of \ref{it:2side2} above, and $n$, a point where periodicity of $S$ begins. Without loss of generality we may assume that $n\geq n_1$.
We prove that $a^\ast \leq v_{\PE(S)}$. To do so, let 
\begin{align*}
    u = (-a, qp, qp+b) \in \bigcup_{a\geq 0, b\geq n}\bigr(\PE(S)\cap \P_{a,b}\bigr),
\end{align*}
the union being non-empty by Lemma~\ref{lem:NonZeroNums}.
By Claim~\ref{cl:2side2} \ref{it:cl2side2-2}, there is an idempotent $f\in E(S)\setminus\PIG(S)$ such that $f\lambda\leq a+qp$. By definition, $a^\ast$ is the least $\lambda$-value of elements in $\inv(S)^\circ$; note that $E(\inv(S)^\circ) = E(S)\setminus\PIG(S)$. In particular, $a^\ast \leq f\lambda \leq a+qp$. Since $u$ was chosen arbitrarily, we have $a^\ast \leq v_{\PE(S)}$, as required. 
That $b^\ast \leq h_{\PE(S)}$ is proved similarly using Claim~\ref{cl:2side2}~\ref{it:cl2side2-1}.
The proofs of $a^\ast \leq v_{\PE(S^{-1})}$ and $b^\ast \leq h_{\PE(S^{-1})}$ are dual,
completing the proof of the forward direction of Theorem  \ref{thm:2side}.

\noindent $(\Leftarrow)$ Let $Y\subseteq\FI_1$ be a finite set such that $\inv(S) = \langle Y\rangle$ (generated as a semigroup). By Lemma~\ref{lem:EuniPlainSamePIG} and Proposition~\ref{prop:JonesTrotterPIG}, $\PIG(S)$ is finite. Thus, $E(S)$ is $p$-periodic from some $n_E$ by Lemma~\ref{lem:PeriodicSemilattice}. Without loss of generality, we assume $n$ in \ref{it:2side2} is large enough: 
\begin{align*}
    n \geq \max\{n_E, a^\ast + p +b_{\alpha}, a_{\beta}+p+b^\ast\}.
\end{align*}

We show that there is a finite set $A$ such that $\PE(S)\subseteq \langle A\rangle$. Note that $(\bigcup_{a,b\geq n}\PE_{a,b})\cap S \subseteq \inv(S)$ by Lemma~\ref{lem:LargeabInv}. Recall that each $\PE_{a,b}\cong\mathbb{N}_0$ if $a+b>0$ and $\PE_{0,0}\cong\mathbb{N}$ (Proposition~\ref{prop:ChoRuskuc}). For each $0\leq a, b< n+\nu$, let $Z_{a,b}$ be a finite generating set of $\PE_{a,b}\cap S$ (such a set exists by Proposition~\ref{prop:SitSiu}). Define $Z$ to be $\bigcup_{0\leq a,b<n+\nu}Z_{a,b}$. We claim that
\begin{align}
    \PE(S)\subseteq \langle Y, Z\rangle. \label{eq:backwardclaim}
\end{align}

To prove the claim, let $u = (-a, qp, qp+b) \in \PE(S)$. We may assume $qp>0$. Further, suppose $0\leq a< n+\nu$ and $b\geq n+\nu$. Since $\PE(S)$ is $\nu$-periodic from $n$ and $b\geq n+\nu$, there exist $u_1 = (-a, qp, qp+b_1)\in \PE(S)$ where $n\leq b_1 <n+\nu$. Note $u\alpha^{-q} = (\max(a, a^\ast), \, 0,\, qp+b)\in E(S)$. Due to the $p$-periodicity of $E(S)$ and $b\geq n+\nu > n_E$, we have $e = (-\max(a, a^\ast),\,0,\, b)\in E(S)$. By the assumption \ref{it:2side3}, $e\lambda \leq a+qp$. Then, $u = u_1e\in\langle Y, Z\rangle$. Similarly, if $a\geq n+\nu$ and $0\leq b<n+\nu$, one can show that $u\in\langle Y,Z\rangle$. This completes the claim (\ref{eq:backwardclaim}). A dual argument shows $\PE(S^{-1})\subseteq \langle Y, T\rangle$ for some finite $T\subseteq \FI_1$. Therefore, $S = \langle Y, Z, T^{-1}\rangle$, completing the proof of the backward direction of the theorem.
\end{proof}

We conclude this section with an example of a two-sided subsemigroup of $\FI_1$ that fails to satisfy \ref{it:2side3} of Theorem~\ref{thm:2side}, and hence is not finitely generated. 
It is worth noting that this example is built as the intersection of two finitely generated two-sided subsemigroups of $\FI_1$, in contrast
with Example \ref{ex:HowsonCounter1}.

\begin{ex}
\label{ex:2sideIntersect}
Let $S$ and $T$ be subsemigroups of $\FI_1$ defined as
\begin{align*}
    S := \langle (-2, 4, 6),\, (-9, -4, 10),\, (-12, -4, 2)\rangle,\;\;\; T := \langle (-2, 4, 6),\, (-8, -4, 10),\, (-12, -4, 2)\rangle.
\end{align*}
Clearly, $S\mu = T\mu = (S\cap T)\mu = 4\mathbb{Z}$. 

Since $S$ and $T$ satisfy condition \ref{it:2side1} in Theorem~\ref{thm:2side} and since $\inv(S)\cap \inv(T) = \inv(S\cap T)$, it follows that $S\cap T$ also satisfies \ref{it:2side1} of Theorem~\ref{thm:2side} using the inverse semigroup Howson property of $\FI_1$ (Theorem~\ref{thm:JonesTrotter}). Both $S$ and $T$ satisfy condition \ref{it:2side2} of Theorem~\ref{thm:2side}. Note that for each $S$ and $T$ we can choose its periodicity and the starting point of the periodic behaviour sufficiently large (Remark~\ref{rmk:PeriodicSets}). Thus, condition \ref{it:2side2} of Theorem~\ref{thm:2side} also holds for $S\cap T$. However, in the following claims we shall show that $S\cap T$ does not satisfy condition \ref{it:2side3} in Theorem~\ref{thm:2side}; in particular, we will prove that $a_{S\cap T}^\ast > v_{\PE(S\cap T)}$.

\begin{claim}
\label{cl:2sideIntersect1}
Let $e\in E(S)$ be such that $e = x_1\dots x_m$ where $x_i\in\{(-2, 4, 6),\, (-9, -4, 10),\, (-12, -4, 2)\}$ for $i=1, \dots, m$. Suppose no proper initial subword of the product is an idempotent. Then, $e\lambda$ is of the form $5+4k$ or $8+4k$ for some $k\in\mathbb{N}_0$.
\end{claim}
\begin{proof}
We proceed by induction on the length $m$ of the product. Observe that $m$ must be even. Consider the base case $m=2$. The following direct calculations deal with the base case: 
\begin{align*}
    &e = (-2, 4, 6)(-9, -4, 10) = (-5, 0, 10); && e=(-2,4,6)(-12,-4,2) = (-8, 0, 6);\\
    &e= (-9, -4, 10)(-2, 4, 6) = (-9, 0, 10); &&e=(-12,-4,2)(-2,4,6) = (-12,0,2).
\end{align*}

For the inductive step, let $m\geq 2$ and assume that the statement holds for products of smaller length.
Using the assumption about proper prefixes of $x_1\dots x_m$, 
observe that
the sign of each $(x_1\dots x_i)\mu$ ($i\in [1,\dots m-1]$) is the same as the sign of $x_1\mu$
and that $x_1\mu = -(x_m\mu)$.
Therefore, $f:=x_2\dots x_{m-1}$ is an idempotent. We decompose $f$ into $f=f_1\dots f_l$ where each $f_i$ is an idempotent satisfying the assumption of the claim. Thus, $f\lambda = \max(f_1\lambda, \dots, f_l\lambda)$ is of the form $5+4k$ or $8+4k$ for some $k\in\mathbb{N}_0$ since each $f_i\lambda$ is by induction. The following are four cases for $e$ depending on the value of $x_1$ and $x_m$:
\begin{bulletmidsentence}
    \item if $e = (-2, 4, 6)f(-9,-4,10)$, then $e\lambda = \max(2, f\lambda -4, 5)$;
    \item if $e = (-2, 4, 6)f(-12, -4, 2)$, then $e\lambda = \max(2, f\lambda-4, 8)$;
    \item if $e = (-9, -4, 10)f(-2, 4, 6)$, then $e\lambda = \max(9, f\lambda + 4, 6)$;
    \item if $e = (-12, -4, 2)f(-2, 4, 6)$, then $e\lambda = \max(12, f\lambda+4, 6)$.
\end{bulletmidsentence}
This completes the induction.
\end{proof}

\begin{claim}
\label{cl:2sideIntersect2}
For all $e\in E(S)$, $e\lambda$ is of the form $5+4k$ or $8+4k$ for some $k\in\mathbb{N}_0$.
\end{claim}

\begin{proof}
This follows from a decomposition of $e$ into $e_1\dots e_m$ such that each $e_i$ satisfies the assumption of Claim~\ref{cl:2sideIntersect1}.
\end{proof}

\begin{claim}
\label{cl:2sideIntersect3}
For all $k\in\mathbb{N}_0$, $(-2, 4, 18+4k)\in S\cap T$.
\end{claim}
\begin{proof}
This is proved by a straightforward induction, where the base case ($k=0$) is established by $
(-2, 4, 18) = (-2, 4, 6)^2(-9,-4,10) = (-2, 4, 6)^2(-8, -4, 10)$.
\end{proof}
We now prove that $a^\ast_{S\cap T} > v_{\PE(S\cap T)}$. Note that $e\lambda$ is even for every $e\in E(T)$. Thus, $e\lambda\geq 8$ for $e\in E(S\cap T) $ by Claim~\ref{cl:2sideIntersect2}. Recall that $a_{S\cap T}^\ast$ is the least $\lambda$-value of elements in $E(S\cap T)\setminus\PIG(S\cap T)$. Thus, $a_{S\cap T}^\ast \geq 8$. But, Claim~\ref{cl:2sideIntersect3} implies that $v_{\PE(S\cap T)}\leq 6$. This shows 
that condition \ref{it:2side3} of Theorem~\ref{thm:2side} fails to hold for $S\cap T$, and hence $S\cap T$ is not finitely generated.
\end{ex}

\section{Finitely Generated Intersection Problem}
\label{sec:FGIP}

In this section, we establish the second main result of the paper, Theorem \ref{thmB:FGIP}, which states that the FGIP for $\FI_1$ as a semigroup is decidable.

The poof is an application of characterisation theorems from Section~\ref{sec:Character} combined with results by Silva~\cite{SilvaRationalSubset}, in particular, the \emph{Cut and Paste Lemma} from that paper. The latter results are concerned with rational subsets of the monogenic free inverse monoid $\FIM_1$. The theory of rational subsets is customarily set out within the context of monoids, rather than semigroups. We introduce necessary results with respect to rational subsets in $\FIM_1$ and then quickly return to subsemigroups of $\FI_1$. 

A subset $L$ of a monoid $M$ is said to be \emph{rational} 
if it can be written as a (possibly empty) \emph{rational expression} over elements of $M$, that is, 
if it can be obtained from singleton subsets of $M$ by applying a finite (possibly empty) sequence of the following operations:
\begin{bulletmidsentence}
    \item union: $(A,B)\mapsto A\cup B$;
    \item product: $(A,B)\mapsto AB := \{ab: a\in A, b\in B\}$;
    \item (Kleene) star: $A \mapsto A^\ast := \bigcup_{n\in\mathbb{N}_0}A^n$.
\end{bulletmidsentence}
When we speak about algorithmic properties of rational subsets, it will always be understood that the input consists of such a
rational
expression for the subset in question.
A comprehensive introduction to rational subsets can be found in \cite{Berstel}.

\begin{lem}[{Cut and Paste \cite[Lemma 4.2]{SilvaRationalSubset}}]
\label{lem:CutPaste}
Let $L$ be a rational subset of $\FIM_1$. Then, there exist computable natural numbers $n \geq \nu \geq 1$ such that the following three statements hold for all $(-a, x, x+b) \in \PE$:
\begin{thmenumerate}
    \item \label{it:CP1} when $a\geq n$: $(-a, x, x+b)\in L$ if and only if $(-(a-\nu), x, x+b)\in L$;
    \item \label{it:CP2} when $x\geq n$: $(-a, x, x+b)\in L$ if and only if $(-a, x-\nu, x-\nu+b)\in L$;
    \item \label{it:CP3} when $b\geq n$: $(-a, x, x+b)\in L$ if and only if $(-a, x, x+b-\nu)\in L$. \qed
\end{thmenumerate}
\end{lem}

\begin{lem}[{\cite[Corollary 4.3]{SilvaRationalSubset}}]
\label{lem:CutPasteCor}
Let $L$ be a rational subset of $\FIM_1$. If $n$ and $\nu$ are natural numbers from the Cut and Paste Lemma (\ref{lem:CutPaste}), then the statement of the Cut and Paste Lemma remains true when $n$ and $\nu$ are replaced by $n'$ and $m\nu$, respectively, where $m,n'\in\mathbb{N}$ with $n'\geq n+m\nu$. \qed
\end{lem}

In order to state the next result from \cite{SilvaRationalSubset},
we define some subsets of $\PE$. As it is still about rational subsets of $\FIM_1$, we 
take $\FIM_1$ to be $\FI_1\cup\{(0,0,0)\}$,  and add $(0,0,0)$ to $\PE$. 
Let $n,\nu\in\mathbb{N}$; for technical reasons in later arguments, assume $n>\nu$. Define:
\begin{align*}
    &\W_{n} := \{(-a, x, x+b)\in \PE : a, x, b < n\};\\
    &\V_{n,\nu}^1 := \{(-a, x, x+b) \in \W_{n} : n-\nu\leq a <n\};\\
    &\V_{n,\nu}^2 := \{(-a, x, x+b)\in \W_{n} : n-\nu \leq  x <n\};\\
    &\V_{n, \nu}^3 := \{(-a, x, x+b) \in \W_{n} : n-\nu \leq b <n \};\\
    &\V_{n,\nu} := \V_{n,\nu}^1\cup \V_{n,\nu}^2 \cup \V_{n,\nu}^3.
\end{align*}
Note that $\V_{n,\nu}\subseteq \W_n$ by definition.

\begin{lem}[{\cite[Theorem 5.1]{SilvaRationalSubset}}]
\label{la:SilvaThm}
Let $L$ be a rational subset of $\FIM_1$. Then, there exist $n > \nu \geq 1$ such that the following hold: 
\begin{thmenumerate}
    \item \label{it:SilvaThm1} the Cut and Paste Lemma (\ref{lem:CutPaste}) holds with respect to $n, \nu$;
    \item \label{it:SilvaThm2} $u\in \PE(L)$ if and only if either $u\in L \cap\W_{n}$ or there is $u'\in L\cap \V_{n,\nu}$ that is obtained from $u$ by applying the Cut and Paste Lemma. \qed
\end{thmenumerate}
\end{lem}

\begin{rmk}
\label{rmk:SilvaThm}
\begin{rmkenumerate}
    \item Suppose $u = (-a, x, x+b)\in \PE(L)$. Either $a,x,b<n$ and so $u\in\W_{n}$, or we apply \ref{it:CP1}, \ref{it:CP2}, \ref{it:CP3} of the Cut and Paste Lemma $i,j,k$ times, respectively, and obtain an element $u'\in L\cap \V_{n,\nu}$. 
In particular, if $a\geq n$, we apply \ref{it:CP1} of the Cut and Paste Lemma $i\geq 1$ times
until we reach $n-\nu\leq a-i\nu <n$. Similarly, if $x\geq n$ (resp. $b\geq n$), we apply \ref{it:CP2} (resp. \ref{it:CP3}) of the Cut and Paste Lemma $j\geq 1$ (resp. $k\geq1$) times. Notice that the obtained element $u'\in \V_{n,\nu}^1$ (resp. $u'\in \V_{n,\nu}^2$, $u'\in \V_{n,\nu}^3$) if $i\geq 1$ (resp. $j\geq 1$, $k\geq1$).  
    \item A dual statement concerning elements of $\NE(L)$ is stated in \cite{SilvaRationalSubset}. This can be obtained by considering $\PE(L^{-1})$; note $L$ is a rational subset of $\FIM_1$ if and only if $L^{-1}$ is. By Lemma~\ref{lem:CutPasteCor}, we can always compute $n>\nu\geq 1$ so that the statement of Lemma~\ref{la:SilvaThm} holds for both $L$ and $L^{-1}$. We specifically record this for finitely generated subsemigroups of $\FI_1$ below.
\end{rmkenumerate}
\end{rmk}

\begin{lem}
\label{lem:CutPasteFGSubsem}
Let $S$ be a finitely generated subsemigroup of $\FI_1$. Then, there exist natural numbers $n>\nu\geq 1$ such that Lemma~\ref{la:SilvaThm} simultaneously holds for both $S$ and $S^{-1}$ with those parameters.
Furthermore, such $n$ and $\nu$  
can be algorithmically computed from
a generating set of $S$. 
\end{lem}

\begin{proof}
Suppose $S = \langle X\rangle$ where $X\subseteq\FI_1\subseteq\FIM_1$ is finite. 
Then $S = XX^\ast$, a rational expression over  $\FIM_1$, and so 
is a rational subset of $\FIM_1$. Similarly, $S^{-1}$ is a rational subset. The result now follows from Lemmas~\ref{la:SilvaThm} and~\ref{lem:CutPasteCor}. 
\end{proof}

From now on, when we refer to computable natural numbers such that Lemma~\ref{la:SilvaThm} holds for a finitely generated subsemigroup $S$ of $\FI_1$, we mean those natural numbers such that Lemma~\ref{la:SilvaThm} holds for both $S$ and $S^{-1}$ as in Lemma~\ref{lem:CutPasteFGSubsem}.

\begin{rmk}
\label{rem:salient}
Recall the definition of periodic subsets (Definition~\ref{defn:PeriodicSets}). 
This definition was partially motivated by the Cut and Paste Lemma (\ref{lem:CutPaste}), and it may be viewed as a weakening of the latter. In particular, the horizontal periodicity and vertical periodicity, respectively, correspond to \ref{it:CP1} and  \ref{it:CP3} of the Cut and Paste Lemma. In the proof of the forward direction of Theorem~\ref{thm:2side},
condition \ref{it:2side2} can be derived immediately by applying the Cut and Paste Lemma.
\end{rmk}

We require one further result before embarking on the proof of the main result. 

\noindent
\textbf{Membership problem for rational subsets of a monoid:} Let $M$ be a monoid generated by a finite set $A$ and let $\pi\colon A^\ast \to M$ be the natural homomorphism. Given $w\in A^\ast$ and a rational subset $L\subseteq M$, is it decidable whether $w\pi \in L$?

The decidability of membership problem for rational subsets of a monoid $M$ does not depend on the choice of a generating set or associated natural homomorphism. Thus, for our purposes, one may consider $A = \{a, a^{-1}\}$ and the canonical homomorphism $\pi\colon A^\ast \to \FIM_1$ defined by $a\mapsto (0, 1, 1)$, $a^{-1}\mapsto (-1, -1, 0)$.

\begin{thm}[{\cite[Theorem 3]{DiekertLohreyMiller}, \cite[Corollary 3.2]{SilvaRationalSubset}}]
\label{thm:Membership}
The rational subset membership problem for a free inverse monoid of finite rank is decidable.
\end{thm}

We now begin the proof of Theorem \ref{thmB:FGIP}.
While the details are somewhat technical, the underlying idea is quite simple and may be summarised as follows.
We are given two finitely generated subsemigroups $S,T\leq \FI_1$, and we need to decide whether or not their intersection $S\cap T$ is finitely generated.
If at least one of $S$ or $T$ is one-sided, then so is $S\cap T$, and is finitely generated by 
Proposition~\ref{prop:FI1fgintersectpairs}.
So suppose that both $S$ and $T$ are two-sided.
The intersection $S\cap T$ may still be one-sided, as we saw in 
Example \ref{ex:HowsonCounter1}.
We determine whether $S\cap T$ is one-sided or two-sided using
Lemma \ref{lem:FGIPDet2side} below.
When $S\cap T$ is one-sided, say $S\cap T\subseteq \PE$, it will be finitely generated if and only if it intersects only finitely many sets $\PE_{a,b}$ by 
Theorem~\ref{thm:1side}. Whether or not this is the case is decided using Lemma~\ref{lem:FGIP1sideorEmpty} below.
Now consider the case when $S\cap T$ is two-sided.
Then it is finitely generated if and only if it satisfies conditions 
\ref{it:2side1}--\ref{it:2side3} of Theorem \ref{thm:2side}.
In fact, it turns out that \ref{it:2side1} and \ref{it:2side2} are automatically satisfied, and so \ref{it:2side3} is the decider (Lemma \ref{lem:FGIP2sideIntersect1} below).
To decide whether or not \ref{it:2side3} holds we have to 
show how to compute the numbers
$h_{\PE(S\cap T)}, h_{\PE(S^{-1}\cap T^{-1})}, v_{\PE(S\cap T)}, v_{\PE(S^{-1}\cap T^{-1})}$.
This is accomplished using Lemmas \ref{lem:FGIP2sideIntersect2}--\ref{lem:FGIPast2}.

We now start establishing technical details. In the sequence of lemmas below, $S$ and $T$ will be two finitely generated subsemigroups of $\FI_1$.
Also fixed will be natural numbers $n>\nu\geq 1$ such that Lemma~\ref{la:SilvaThm} holds for both $S$ and $T$; this is always possible due to Lemma~\ref{lem:CutPasteCor}.

The following lemma will be used to determine whether the intersection $S\cap T$ is one-sided or two-sided.

\begin{lem}
\label{lem:FGIPDet2side}
With the notation above, the following hold:
\begin{thmenumerate}
    \item \label{it:FGIPDet2side-1} $\P(S)\cap \P(T)\neq \emptyset$ if and only if $S\cap T\cap \P(\W_{n})\neq\emptyset$;
    \item \label{it:FGIPDet2side-2} $\P(S^{-1})\cap \P(T^{-1}) \neq \emptyset$ if and only if $S^{-1}\cap T^{-1}\cap \P(\W_{n})\neq \emptyset$.
\end{thmenumerate}
\end{lem}

\begin{proof}
We prove \ref{it:FGIPDet2side-1}; the proof of \ref{it:FGIPDet2side-2} is dual. The backward direction is trivial. For the forward direction, suppose $u\in \P(S)\cap \P(T)$. Since $n> \nu \geq 1$ are common for both $S$ and $T$, Lemma~\ref{la:SilvaThm} \ref{it:SilvaThm2} implies that either $u\in \W_{n}$ or there is $u'\in S\cap T\cap \V_{n,\nu}$ that is obtained from $u$ by applying the Cut and Paste Lemma to both $S$ and $T$. The result immediately follows if the former holds. Assume the latter. Note that the second component of $u'$ is obtained by applying \ref{it:CP2} of the Cut and Paste Lemma. If \ref{it:CP2} is applied zero times, then $u'\mu = u\mu$; otherwise, $u'\in \V_{n,\nu}^2$ (Remark~\ref{rmk:SilvaThm}). Notice that $\V_{n,\nu}^2\subseteq \P(\W_{n})$. Thus, in either case, $u'\in \P(\W_{n})$, completing the proof.
\end{proof}

The following is concerned with the case when $S\cap T$ is one-sided or empty:

\begin{lem}
\label{lem:FGIP1sideorEmpty}
Assume $S\cap T\subseteq\PE$. Then, $S\cap T$ has a non-empty intersection with infinitely many sets $\PE_{a,b}$ if and only if $S\cap T\cap (\V_{n,\nu}^1\cup \V_{n,\nu}^3) \neq \emptyset$.
\end{lem}

\begin{proof}
($\Rightarrow$) Using the assumption, pick $u = (-a, x, x+b) \in S\cap T$ with $a\geq n$ or $b\geq n$. Assume the former. Then, Lemma~\ref{la:SilvaThm} \ref{it:SilvaThm2} yields an element $u'\in S\cap T\cap \V_{n,\nu}$ that is obtained from applying the Cut and Paste Lemma to both $S$ and $T$. Since $a\geq n$, \ref{it:CP1} of the Cut and Paste Lemma is applied at least once, and hence $u'\in \V_{n,\nu}^1$ (Remark~\ref{rmk:SilvaThm}). Similarly, if $b\geq n$, we obtain an element in $S\cap T\cap \V_{n,\nu}^3$.

\noindent ($\Leftarrow$)
Suppose $(-a, x, x+b)\in S\cap T\cap \V_{n,\nu}^1$.
Using \ref{it:CP1} of the Cut and Paste Lemma $i\in\mathbb{N}$ times in both $S$ and $T$ yields
$(-a-i\nu,x,x+b)\in S\cap T\cap \PE_{a+i\nu,b}$.
The case $(-a, x, x+b)\in S\cap T\cap \V_{n,\nu}^3$ is analogous using 
\ref{it:CP3} of the Cut and Paste Lemma.
\end{proof}

The remaining lemmas deal
with the case where $S\cap T$ is two-sided; this assumption is implicit throughout.

\begin{lem}
\label{lem:FGIP2SideIntersectPeriod}
The subsemigroup $S\cap T$ is periodic from $n-\nu$ with period $\nu$.
\end{lem}
\begin{proof}
This is obvious since both $S$ and $T$ are periodic from $n-\nu$ with period $\nu$.
\end{proof}

\begin{lem}
\label{lem:FGIP2sideIntersect1}
The intersection $S\cap T$ is finitely generated if and only if it satisfies Theorem~\ref{thm:2side}\ref{it:2side3}, that is, $a_{S\cap T}^\ast \leq \min(v_{\PE(S\cap T)}, v_{\PE(S^{-1}\cap T^{-1})}) \text{ and } b_{S\cap T}^\ast \leq \min(h_{\PE(S\cap T)}, h_{\PE(S^{-1}\cap T^{-1})})$.
\end{lem}

\begin{proof}
The result follows from Theorem~\ref{thm:2side} upon verifying that the conditions \ref{it:2side1} and \ref{it:2side2} are satisfied. Condition \ref{it:2side2} is established by Lemma~\ref{lem:FGIP2SideIntersectPeriod}. For \ref{it:2side1}, Theorem~\ref{thm:2side} applied to $S$ and $T$ implies that both $\inv(S)$ and $\inv(T)$ are finitely generated. Since $\inv(S\cap T) = \inv(S)\cap \inv(T)$, the inverse semigroup Howson property of $\FI_1$ (Theorem~\ref{thm:JonesTrotter}) implies that $\inv(S\cap T)$ is finitely generated, as required.
\end{proof}

In order to determine whether 
$S\cap T$ satisfies the condition of Lemma~\ref{lem:FGIP2sideIntersect1},
we need to be able to compute the numbers $h_{\PE(S\cap T)}, h_{\PE(S^{-1}\cap T^{-1})}, v_{\PE(S\cap T)}, v_{\PE(S^{-1}\cap T^{-1})}$; see Definition \ref {defn:HorizonVertNumbers}.
Observe that none of these numbers are zero (Lemmas~\ref{lem:NonZeroNums} and \ref{lem:FGIP2SideIntersectPeriod}).

\begin{lem}
\label{lem:FGIP2sideIntersect2}
All of the following sets are non-empty: 
\[
S\cap T\cap \P(\V_{n,\nu}^1),\quad\; S\cap T\cap\P(\V_{n,\nu}^3),\quad\; S^{-1}\cap T^{-1}\cap \P(\V_{n,\nu}^1),\quad\; S^{-1}\cap T^{-1}\cap \P(\V_{n,\nu}^3).
\]
\end{lem}

\begin{proof}
We show $S\cap T\cap \P(\V_{n,\nu}^1) \neq \emptyset$; 
the proofs for the other sets
are similar. 
We have already mentioned that $h_{\PE(S\cap T)}\neq 0$, which means 
$\bigcup\{\PE(S\cap T)\cap \P_{a,b} : a\geq n-\nu, b\geq 0\}\neq \emptyset$
(Definition \ref{defn:HorizonVertNumbers}).
Let $u$ be an arbitrary element of this set. Again, by Lemma~\ref{la:SilvaThm} \ref{it:SilvaThm2}, either $u\in \W_{n}$, in which case the result is immediate, or there is $u'\in S\cap T\cap \V_{n,\nu}$ obtained from applying the Cut and Paste Lemma to both $S$ and $T$. A routine argument then shows $u'\in \P(\V_{n,\nu}^1)$.
\end{proof}

The next result shows that to find the numbers $h_{\PE(S\cap T)} $ and $v_{\PE(S\cap T)}$ it is sufficient to inspect the intersections of $S\cap T$ with $\P(\V_{n, \nu}^1)$ and $ \P(\V_{n,\nu}^3)$ respectively.

\begin{lem}
\label{lem:FGIPhvPositive}
We have
\begin{align*}
    h_{\PE(S\cap T)} &= \min\bigl\{x+b : (-a, x, x+b) \in S\cap T\cap \P(\V_{n, \nu}^1)\bigr\},\\
    v_{\PE(S\cap T)} &= \min\bigl\{a+x : (-a, x, x+b)\in S\cap T\cap \P(\V_{n,\nu}^3)\bigr\}.
\end{align*}
\end{lem}

\begin{proof}
We prove the second assertion, and the first can be done analogously. Let $v^+ := \min\{a+x : (-a, x, x+b)\in S\cap T\cap \P(\V_{n,\nu}^3)\}$; this is well-defined by Lemma~\ref{lem:FGIP2sideIntersect2}. 

Note that 
\[S\cap T\cap \P(\V_{n,\nu}^3) \subseteq \bigcup\{\PE(S\cap T)\cap \P_{a,b} : a\geq 0, b\geq n-\nu\}.
\]
Thus, $v_{\PE(S\cap T)}\leq v^+$.

Let $u \in \bigcup\{\PE(S\cap T)\cap \P_{a,b} : a\geq 0, b\geq n-\nu\}
$. By Lemma~\ref{la:SilvaThm} \ref{it:SilvaThm2}, either $u\in \W_{n}$, in which case $u\in S\cap T\cap \P(\V_{n,\nu}^3)$ and $v^+\leq u\lambda + u\mu$, or there is $u'\in S\cap T\cap \V_{n,\nu}$ that is obtained from $u$ by applying the Cut and Paste Lemma to both $S$ and $T$. It is routine to show that $u'\in \P(\V_{n,\nu}^3)$. Note that $u'\lambda\leq u\lambda$ and $u'\mu\leq u\mu$. Thus, $v^+\leq u\lambda + u\mu$. Since $u$ is arbitrary, this implies $v^+\leq v_{\PE(S\cap T)}$. Thus $v_{\PE(S\cap T)}=v^+$, as desired.
\end{proof}

The following is an obvious dual of Lemma~\ref{lem:FGIPhvPositive}, and
we omit the proof.

\begin{lem}
\label{lem:FGIPhvNegative}
We have
\begin{align*}
    h_{\PE(S^{-1}\cap T^{-1})} &= \min\bigl\{x+b : (-a, x, x+b)\in S^{-1}\cap T^{-1}\cap \P(\V_{n,\nu}^1)\bigr\};\\
    v_{\PE(S^{-1}\cap T^{-1})} &= \min\bigl\{a+x : (-a, x, x+b)\in S^{-1}\cap T^{-1}\cap \P(\V_{n,\nu}^3)\bigr\}.
    \tag*{\qed}
\end{align*}
\end{lem}

The final three lemmas are concerned with the values of 
$h_{\PE(S\cap T)},\ h_{\PE(S^{-1}\cap T^{-1})}, \ v_{\PE(S\cap T)}$, \ $v_{\PE(S^{-1}\cap T^{-1})}$
in relation to $a_{S\cap T}^\ast$ and $b_{S\cap T}^\ast$. 

\begin{lem}
\label{lem:InfManyIdems}
Let $U$ be a subsemigroup of $\FI_1$. Suppose $e\in E(U)\setminus\PIG(U)$. Then, there are infinitely many idempotents $f_1, f_2\in E(U)\setminus\PIG(U)$ such that $f_1\lambda = e\lambda$ and $f_2\rho = e\rho$.
\end{lem}
\begin{proof}
By Lemma~\ref{lem:EuniPIGGreen}, there is $u\in \inv(U)$ such that $uu^{-1} = e$. Without loss of generality, assume $u$ is positive. Then, $(u^iu^{-i})\lambda = e\lambda$ and $((eu^{-1})^i(eu^{-1})^{-i})\rho = e\rho$ for each $i\in \mathbb{N}$. 
\end{proof}

\begin{lem}
\label{lem:FGIPast1}
The following are equivalent:
\begin{thmenumerate}
    \item \label{it:FGIPast-1} $a_{S\cap T}^\ast \leq \min(v_{\PE(S\cap T)},\, v_{\PE(S^{-1}\cap T^{-1})})$;
    \item \label{it:FGIPast-2} there is an idempotent $e\in E(S\cap T)\setminus\PIG(S\cap T)$ such that 
    \[
    e\lambda \leq \min(v_{\PE(S\cap T)},\, v_{\PE(S^{-1}\cap T^{-1})});
    \]
    \item \label{it:FGIPast-3} there is an idempotent $f_1\in S\cap T\cap \V_{n,\nu}^3$ such that 
    \[
    f_1\lambda \leq \min(v_{\PE(S\cap T)},\, v_{\PE(S^{-1}\cap T^{-1})}).
    \]
\end{thmenumerate}
\end{lem}
\begin{proof}
We have already observed that $a_{S\cap T}^\ast$ is the least $\lambda$-value of triples in $E(S\cap T)\setminus \PIG(S\cap T)$; this shows the equivalence between \ref{it:FGIPast-1} and \ref{it:FGIPast-2}.

\noindent \ref{it:FGIPast-2} $\Rightarrow$ \ref{it:FGIPast-3}. 
Suppose $e \in E(S\cap T)\setminus\PIG(S)$ as in \ref{it:FGIPast-2}.
We may assume $e\rho \geq n$ from Lemma~\ref{lem:InfManyIdems}.
Once again, Lemma~\ref{la:SilvaThm} \ref{it:SilvaThm2} yields an idempotent $f_1\in S\cap T\cap \V_{n,\nu}$ that is obtained from $e$ by applying \ref{it:CP1} and \ref{it:CP3} of the Cut and Paste Lemma to both $S$ and $T$. Note that \ref{it:CP3} is applied at least once since $e\rho \geq n$, and so $f_1\in \V_{n,\nu}^3$ (Remark~\ref{rmk:SilvaThm}). It is clear that $f_1\lambda \leq e\lambda$.

\noindent \ref{it:FGIPast-3} $\Rightarrow$ \ref{it:FGIPast-2}. Suppose $f_1 = (-a_1, 0, b_1)$ is an idempotent as in
\ref{it:FGIPast-3}. We apply \ref{it:CP3} of the Cut and Paste Lemma infinitely many times to both $S$ and $T$ and yield $(-a_1, 0, b_1+l\nu)\in S\cap T$ for each $l\in\mathbb{N}$. Since $\PIG(S\cap T)$ is finite by Lemmas~\ref{lem:FGIP2sideIntersect1}, \ref{lem:EuniPlainSamePIG} and Proposition~\ref{prop:JonesTrotterPIG}, the result follows.
\end{proof}

The following is analogous to Lemma~\ref{lem:FGIPast1}; we omit the proof. 

\begin{lem}
\label{lem:FGIPast2}
The following are equivalent: 
\begin{thmenumerate}
    \item \label{it:FGIPast2-1} $b_{S\cap T}^\ast \leq \min(h_{\PE(S\cap T)},\, h_{\PE(S^{-1}\cap T^{-1})})$;
    \item \label{it:FGIPast2-2} there is an idempotent $e\in E(S\cap T) \setminus\PIG(S\cap T)$ such that 
    \[
    e\rho \leq \min(h_{\PE(S\cap T)},\, h_{\PE(S^{-1}\cap T^{-1})});
    \]
    \item \label{it:FGIPast2-3} there is an idempotent $f_2 \in S\cap T\cap \V_{n,\nu}^1$ such that 
    \[
    f_2\rho \leq \min(h_{\PE(S\cap T)},\, h_{\PE(S^{-1}\cap T^{-1})}). \tag*{\qed}
    \]
\end{thmenumerate}
\end{lem}

We are now ready to prove the main result of this section:

\begin{proof}[Proof of Theorem~\ref{thmB:FGIP}]
Let $S = \langle X\rangle$ and $T = \langle Y \rangle$ where $X,Y\subseteq\FI_1$ are finite. We decide whether $S\cap T$ is finitely generated or not in the following steps. 

\noindent \textbf{Step 1:} \textit{Determine whether $S$  is one-sided or two-sided.} 
Compute the set $X\mu$. If it contains both positive and negative numbers, $S$ is two-sided; otherwise it is one-sided.

\noindent \textbf{Step 2:} \textit{Determine whether $T$  is one-sided or two-sided.} 
This is analogous to Step 1.

\noindent \textbf{Step 3:} \textit{First branching.} 
If at least one of $S$ or $T$ is one-sided,
then $S\cap T$ is finitely generated by Proposition~\ref{prop:FI1fgintersectpairs}, and the algorithm stops. Otherwise proceed to Step 4.

So, from this point onwards both $S$ and $T$ are two-sided.

\noindent\textbf{Step 4:} \textit{Determine whether $S\cap T$ is one-sided or two-sided.}
Compute natural numbers $n>\nu\geq 1$ such that Lemma~\ref{la:SilvaThm} holds for both $S$ and $T$. Now, we have the following sequence of equivalent statements:
\begin{align*}
    S\cap T \text{ is two-sided}& \Leftrightarrow \P(S)\cap \P(T)\neq\emptyset \text{ and } \N(S)\cap \N(T)\neq\emptyset \\
    &\Leftrightarrow \P(S)\cap \P(T)\neq\emptyset \text{ and } \P(S^{-1})\cap \P(T^{-1})\neq\emptyset\\
    &\Leftrightarrow S\cap T\cap \P(\W_{n})\neq\emptyset \text{ and } S^{-1}\cap T^{-1}\cap \P(\W_{n})\neq\emptyset && \text{(Lemma~\ref{lem:FGIPDet2side}).}
\end{align*}
The final statement is algorithmically decidable since $\P(\W_{n})$ is finite, $S, T, S^{-1}, T^{-1}$ are rational subsets of $\FIM_1$, and the rational subset membership problem for $\FIM_1$ is decidable (Theorem~\ref{thm:Membership}).

\noindent\textbf{Step 5:}
\textit{Second branching.}
If $S\cap T$ is one-sided, proceed to Step 6;  otherwise, proceed to Step 7.

\noindent\textbf{Step 6:} 
\textit{When $S\cap T$ is one-sided, decide whether or not it is finitely generated.}
Suppose $S\cap T\subseteq \PE$.
Then:
\begin{align*}
    &S\cap T \text{ is finitely generated}\\
    \Leftrightarrow\ & S\cap T \text{ has a non-empty intersection with finitely many $\PE_{a,b}$} && \text{(Theorem~\ref{thm:1side})}\\
    \Leftrightarrow\ & S\cap T\cap (\V_{n,\nu}^1\cup \V_{n,\nu}^3) = \emptyset && \text{(Lemma~\ref{lem:FGIP1sideorEmpty}).}
\end{align*}
Since $\V_{n,\nu}^1\cup \V_{n,\nu}^3$ is finite, the final statement can be algorithmically decided using the rational subset membership problem for $\FIM_1$.

\noindent\textbf{Step 7:} \textit{When $S\cap T$ is two-sided, decide whether or not it is finitely generated.}
By Lemma~\ref{lem:FGIPhvPositive}, 
\[
h_{\PE(S\cap T)} = \min\{x+b:(-a, x, x+b)\in S\cap T\cap \P(\V_{n,\nu}^1)\}.
\]
Since $\P(\V_{n,\nu}^1)$ is finite, we effectively compute $h_{\PE(S\cap T)}$ using the rational subset membership problem. Similarly, we determine $h_{\PE(S^{-1}\cap T^{-1})},$ $v_{\PE(S\cap T)}, v_{\PE(S^{-1}\cap T^{-1})}$ using Lemmas \ref{lem:FGIPhvPositive} and \ref{lem:FGIPhvNegative}.
Now, we have: 
\begin{align*}
    &S\cap T \text{ is finitely generated}\\
    \Leftrightarrow\ & a_{S\cap T}^\ast \leq \min(v_{\PE(S\cap T)}, v_{\PE(S^{-1}\cap T^{-1})}) \text{ and } b_{S\cap T}^\ast \leq \min(h_{\PE(S\cap T)}, h_{\PE(S^{-1}\cap T^{-1})})\\
    \Leftrightarrow\ & \text{there are idempotents $f_1\in S\cap T\cap \V_{n,\nu}^3$ and $f_2\in S\cap T \cap \V_{n,\nu}^1$ such that}\\
    &\quad\;\, f_1\lambda \leq \min(v_{\PE(S\cap T)}, v_{\PE(S^{-1}\cap T^{-1})}) \text{ and } f_2\rho \leq \min(h_{\PE(S\cap T)}, h_{\PE(S^{-1}\cap T^{-1})}).
\end{align*}
The first equivalence is due to Lemma~\ref{lem:FGIP2sideIntersect1}; the second is due to Lemmas~\ref{lem:FGIPast1} and \ref{lem:FGIPast2}. The final statement is
algorithmically decidable using the rational subset membership problem for $\FIM_1$ since both $\V_{n,\nu}^1$ and $\V_{n,\nu}^3$ are finite. This completes Step 7 and the proof.
\end{proof}

\section{Remarks and open problems}
\label{sec:RemarkOpen}

All results (including examples) in this paper hold if one replaces semigroups with monoids. 

Silva~\cite[Theorem 7.3]{SilvaRationalSubset} showed that a submonoid of the monogenic free inverse monoid $\FIM_1$ is rational if and only if it is finitely generated. By changing semigroups to monoids in our examples \ref{ex:HowsonCounter1} and \ref{ex:2sideIntersect}, we conclude that rational subsets of $\FIM_1$ are not closed under intersection. Thus, rational subsets are not closed under complementation; this is also shown in \cite[Example 5.5]{SilvaRationalSubset}.

We note that our proof of the decidability of FGIP for $\FI_1$ does not indicate its computational complexity. Also, the proof depends on other results, particularly the Cut and Paste Lemma and the rational subset membership problem for $\FIM_1$. It is known that the computational complexity of the rational subset membership problem for a free inverse semigroup is NP (\cite[Theorem 3]{DiekertLohreyMiller}), but the Cut and Paste Lemma does not directly provide complexity bounds for the computation of the parameters $n$ and $\nu$.

\begin{question}
What is the computational complexity of the FGIP for $\FI_1$?
\end{question}

When the intersection of two finitely generated subsemigroups of $\FI_1$ is again finitely generated, we ask the following question: 
\begin{question}
\label{qu:rank}
What is a sharp bound for the rank (i.e. a minimal number of generators) of $S\cap T$ in terms of ranks of $S$ and $T$
if $S, T, S\cap T$ are all finitely generated?
\end{question}

We remark that the FGIP can be defined for any algebra. In particular, we may consider the FGIP for free inverse semigroups as inverse semigroups. Of course, the FGIP for $\FI_1$ as an inverse semigroup is trivial due to the inverse semigroup Howson property (Theorem~\ref{thm:JonesTrotter}). However, as mentioned in the Introduction, a general free inverse semigroup $\FI_n$ of rank $n\geq 2$ does not have the inverse semigroup Howson property (\cite{JonesTrotter}), not to mention the semigroup Howson property.

\begin{question}
Is the FGIP for $\FI_n$ ($n\geq 2$) as an inverse semigroup decidable?
\end{question}

\begin{question}
Is the FGIP for $\FI_n$ ($n\geq 2$) as a semigroup decidable?
\end{question}

We also noted in the Introduction that free semigroups (monoids) of rank $\geq 2$ do not possess the Howson property.
Nonetheless, one can deduce that the FGIP is decidable from general theory of regular languages. Given two finitely generated subsemigroups $S$ and $T$ of a free semigroup (monoid), all of $S$, $T$, $S\cap T$ are regular. Every subsemigroup $U$ of a free semigroup (monoid) admits a unique base (i.e.\ minimal generating set) $U^+\setminus (U^+)^2$ where $U^+ := U\setminus\{1\}$. The base for $S\cap T$ is regular, and whether this base is finite or not is decidable. 
Although the decidability of the FGIP for free semigroups (monoids) can be quickly derived, the study of the rank of the intersection of two finitely generated subsemigroups (submonoids) of a free semigroup (monoid) is not trivial; see, for example, \cite{Karhumaki, GiambrunoRestivo, SinghKrishna} for special cases on this topic.

Free groups are other free objects for which the FGIP for semigroups or for monoids is relevant. Although the FGIP for free groups as groups is trivial due to the result by Howson \cite{Howson}, free groups can be regarded as semigroups or monoids. Notice that free semigroups (monoids) embed into free groups, and hence free groups do not have the (semigroup or monoid) Howson property.

\begin{question}\label{qu:FGIPFG}
Is the FGIP for a free group as a semigroup (monoid) decidable?
\end{question}

Another class of semigroups in which the (semgiroup) Howson property fails, and hence the FGIP is non-trivial, is the class of commutative semigroups. There are commutative semigroups with the Howson property, such as the natural numbers $\mathbb{N}$. However, the following simple example shows that commutative semigroups do not possess the Howson property in general. 

\begin{ex}
\label{ex:commnonHowson}
Consider the direct product $\mathbb{N}_0 \times \mathbb{N}_0\times M$ of monoids $\mathbb{N}_0$ with addition and $M := \{0, 1\} \subseteq \mathbb{N}_0$ with multiplication. Define two finitely generated subsemigroups as 
\[
S := \langle (0, 0, 0), (0, 1, 0), (1, 0, 0)\rangle, \quad T := \langle (0, 0, 1), (0, 1, 1), (1, 0, 1), (1, 1, 0)\rangle.
\]
Observe that $S = \mathbb{N}_0\times\mathbb{N}_0\times\{0\}$ and that $T = (\mathbb{N}\times\mathbb{N}\times\{0\})\cup(\mathbb{N}_0\times\mathbb{N}_0\times \{1\})$. Thus, $S\cap T = \mathbb{N}\times\mathbb{N}\times\{0\} \cong \mathbb{N}\times\mathbb{N}$. It is easy to show that $\mathbb{N}\times\mathbb{N}$ is not finitely generated.
\end{ex}

\begin{question}
\label{open:commHowson}
Which (finitely generated) commutative semigroups have the Howson property? For those commutative semigroups that are not Howson, is the FGIP decidable?
\end{question}

\section*{Acknowledgements} 
The first author is grateful to Ganna Kudryavtseva for her question (Question~\ref{qu:FGIPFG}) and helpful discussions at the Conference on Theoretical and Computational Algebra 2025 in \'{E}vora. The author would also like to thank Carl-Fredrik Nyberg-Brodda for his question (Question~\ref{qu:rank}) at the same conference.

\bibliographystyle{abbrv}
\bibliography{Ref}
	
\end{document}